\documentclass[10pt]{amsart}
\usepackage{amssymb, amsthm, amsmath}
\usepackage[all]{xy}
\usepackage{graphicx}
\usepackage{psfrag}
\usepackage[vcentermath]{youngtab}

\newtheorem{thm}{Theorem}
\newtheorem{lemma}{Lemma}[section]
\newtheorem{prop}[lemma]{Proposition}
\newtheorem{cor}[lemma]{Corollary}
\newtheorem{iprop}{Proposition}
\theoremstyle{definition}

\newtheorem{defn}[lemma]{Definition}
\newtheorem{example}[lemma]{Example}
\newtheorem{remark}[lemma]{Remark}
\newtheorem{claim}{Claim}

\newcommand\A{{\mathbb A}}

\newcommand\C{{\mathbb C}}
\newcommand\Q{{\mathbb Q}}
\newcommand\N{{\mathbb N}}
\newcommand\Eta{H}
\newcommand{\cN}{{\mathcal N}}

\newcommand\IH{{\mathbb H}}
\newcommand{\ti}{\vartheta}
\newcommand{\Ti}{\Theta}

\newcommand\W{{\mathrm W}}

\newcommand\bP{{\mathbb P}}

\newcommand\cC{{\mathcal C}}

\newcommand\X{{\mathrm X}}
\newcommand\Z{{\mathbb Z}}
\newcommand\cP{{\mathcal P}}
\newcommand\cQ{{\mathcal Q}}
\newcommand\cS{{\mathcal S}}

\newcommand\AS{{\mathfrak S}}
\newcommand\DS{{\mathfrak D}}

\newcommand\al{\alpha}

\newcommand\la{\lambda}

\newcommand{\f}{f}
\newcommand{\g}{g}
\newcommand\s{{\sigma}}
\newcommand\ta{{\tau}}

\newcommand\ssm{\smallsetminus}
\newcommand\noin{\noindent}

\newcommand\bull{{\scriptscriptstyle \bullet}}
\newcommand\eqto{\stackrel{\lower1.5pt\hbox{$\scriptstyle\sim\,$}}\to}
\newcommand\ov{\overline}

\newcommand\wh{\widehat}
\newcommand\wt{\widetilde}

\DeclareMathOperator{\SO}{SO}

\DeclareMathOperator{\OG}{OG}

\DeclareMathOperator{\ev}{ev}
\DeclareMathOperator{\HH}{\mathrm{H}}
\DeclareMathOperator{\QH}{\mathrm{QH}}

\DeclareMathOperator{\type}{\mathrm{type}}

\newcommand{\ignore}[1]{}
\newcommand{\pic}[2]{\includegraphics[scale=#1]{#2}}

\psfrag{k}{$k$}
\psfrag{a}{$a$}
\psfrag{b}{$b$}
\psfrag{h}{$h$}
\psfrag{g}{$g$}
\psfrag{x}{$x$}
\psfrag{bh}{$b_h$}
\psfrag{gh}{$g_h$}
\psfrag{bx}{${\mathbf x}$}
\psfrag{(r,c)}{$(r,c)$}
\psfrag{(r',c')}{$(r',c')$}
\psfrag{la1}{$\la^1$}
\psfrag{la2}{$\la^2$}

\begin{document}

\title[A Giambelli formula for even orthogonal Grassmannians]
{A Giambelli formula for even orthogonal Grassmannians}

\date{March 29, 2012}
\author{Anders Skovsted Buch}
\address{Department of Mathematics, Rutgers University, 110
  Frelinghuysen Road, Piscataway, NJ 08854, USA}
\email{asbuch@math.rutgers.edu}

\author{Andrew Kresch}
\address{Institut f\"ur Mathematik,
Universit\"at Z\"urich, Winterthurerstrasse 190,
CH-8057 Z\"urich, Switzerland}
\email{andrew.kresch@math.uzh.ch}

\author{Harry~Tamvakis} \address{Department of Mathematics, University
  of Maryland, 1301 Mathematics Building, College Park, MD 20742, USA}
\email{harryt@math.umd.edu}


\thanks{The authors were supported in part by NSF Grant DMS-0906148
  (Buch), the Swiss National Science Foundation (Kresch), and NSF
  Grant DMS-0901341 (Tamvakis).}

\begin{abstract}
Let $X$ be an orthogonal Grassmannian parametrizing isotropic
subspaces in an even dimensional vector space equipped with a
nondegenerate symmetric form. We prove a Giambelli formula which
expresses an arbitrary Schubert class in the classical and quantum
cohomology ring of $X$ as a polynomial in certain special Schubert
classes. Our analysis reveals a surprising relation between the
Schubert calculus on even and odd orthogonal Grassmannians. We also
study {\em eta polynomials}, a family of polynomials defined using
raising operators whose algebra agrees with the Schubert calculus on
$X$.
\end{abstract}

\maketitle 

\setcounter{section}{-1}

\section{Introduction}
\label{intro}

Consider a complex vector space $V$ of dimension $N$ equipped with a
nondegenerate symmetric form. Choose an integer $m<N/2$ and consider
the Grassmannian $\OG=\OG(m,N)$ parametrizing isotropic
$m$-dimensional subspaces of $V$. Our aim in this paper is to prove a
{\em Giambelli formula} that expresses the Schubert classes on $\OG$
as polynomials in certain special Schubert classes that generate the
cohomology ring $\HH^*(\OG,\Z)$. When $N=2n+1$ is odd, this was the
main result of \cite{BKT2}; what is new here concerns the even case
$N=2n+2$. 

The proof of our main theorem (Theorem \ref{mainthm}) exploits the
weight space decomposition of $\HH^*(\OG(m,2n+2),\Q)$ induced by the
natural involution of the Dynkin diagram of type $\mathrm{D}_{n+1}$.
We require the Giambelli formula for odd orthogonal Grassmannians from
\cite{BKT2} and a similar result for the $(+1)$-eigenspace of
$\HH^*(\OG(m,2n+2),\Q)$, which is the subring generated by the Chern
classes of the tautological vector bundles over $\OG$. These
ingredients combine to establish Theorem \ref{mainthm} thanks to 
a surprising new relation between the cohomology of even and odd
orthogonal Grassmannians (Proposition \ref{prop-1}).

Define nonnegative integers $K$ and $k$ by the equations
\[
K = N-2m = \begin{cases} 2k+1 & \text{if $N$ is odd}, \\
                         2k & \text{if $N$ is even}. 
           \end{cases} 
\]
Observe that $n+k = N-m-1$. An integer partition
$\la=(\la_1,\ldots,\la_\ell)$ is {\em $k$-strict} if no part $\la_i$
greater than $k$ is repeated.  Let $\la$ be a $k$-strict partition
whose Young diagram is contained in an $m \times (n+k)$ rectangle. For
$1 \leq j \leq m$, let
\begin{equation}
\label{peq} 
\ov{p}_j(\la) = N-m + j-\la_j - \#\{i \leq j \mid \la_i+\la_j \geq
K+j-i \text{ and } \la_i > k \} \,,
\end{equation}
and notice that $\ov{p}_j(\la)\neq n+1$ for every $j$ and $\la$.

An isotropic flag is a complete flag $0 = F_0 \subsetneq F_1
\subsetneq \dots \subsetneq F_N = V$ of subspaces of $V$ such that
$F_i = F_j^\perp$ whenever $i+j=N$. For any fixed isotropic flag
$F_\bull$ and any $k$-strict partition $\la$ whose Young diagram is
contained in an $m\times (n+k)$ rectangle, we define a closed
subset $Y_\la=Y_\la(F_\bull) \subset \OG$ by setting
\begin{equation}
\label{Yeq}
Y_\la(F_\bull) = \{ \Sigma \in \OG \mid \dim(\Sigma \cap F_{\ov{p}_j}) 
\geq j \text{ for } 1 \leq j \leq m \} \,.
\end{equation}
If $N$ is odd, the varieties $Y_\la$ are exactly the Schubert
varieties in $\OG$. If $N$ is even, and $k$ is not a part of $\la$,
then $Y_\la$ is again a Schubert variety in $\OG$. Otherwise, $Y_\la$
is a union of two Schubert varieties $X_\la$ and $X'_\la$, which will
be defined below. The algebraic set $Y_\la$ has pure codimension
$|\la|=\sum \la_i$ and determines a class $[Y_{\la}]$ in
$\HH^{2|\la|}(\OG,\Z)$.

Consider the exact sequence of vector bundles over $X=\OG$
\[
0 \to \cS \to V_X \to \cQ \to 0,
\]
where $V_X$ denotes the trivial bundle of rank $N$ and $\cS$ is the
tautological subbundle of rank $m$.  The Chern classes $c_p =
c_p(\cQ)$ of $\cQ$ satisfy 
\begin{equation}
\label{ctoY}
c_p=
\begin{cases}
[Y_p] &\text{if $p \leq k$},\\
2[Y_p] &\text{if $p> k$}.
\end{cases}
\end{equation}

As in \cite{BKT2}, we will express our Giambelli formulas using
Young's raising operators \cite{Y}.  For any integer sequence
$\alpha=(\alpha_1,\alpha_2,\ldots)$ with finite support and $i<j$, we
define $R_{ij}(\alpha) = (\alpha_1, \ldots, \alpha_i+1, \ldots,
\alpha_j-1, \ldots)$.  We also set $c_\alpha = \prod_i c_{\alpha_i}$.
If $R$ is any finite monomial in the $R_{ij}$'s, then set $R\,c_{\al}
= c_{R \,\al}$; we stress that the operator $R$ acts on the subscript
$\al$ and not on the monomial $c_\al$ itself.  Given a $k$-strict
partition $\la$ we define the operator
\begin{equation}
\label{Req}
R^{\la} =\prod (1-R_{ij})\prod_{\la_i+\la_j \geq
K+j-i}(1+R_{ij})^{-1}
\end{equation}
where the first product is over all pairs $i<j$ and the second product
is over pairs $i<j$ such that $\la_i+\la_j \geq K+j-i$.  Let
$\ell_k(\la)$ denote the number of parts $\la_i$ which are strictly
greater than $k$.

\begin{thm}
\label{Chthm}
For any $k$-strict partition $\la$ contained in an $m \times
(n+k)$ rectangle, we have 
$[Y_\la] =2^{-\ell_k(\la)}R^{\la} \, c_{\la}$
in the cohomology ring of $\OG(m,N)$. 
\end{thm}

When $N$ is odd, Theorem \ref{Chthm} is the Giambelli formula for the
Schubert classes on odd orthogonal Grassmannians from \cite[\S
2]{BKT2}; the result for even $N$ is proved along the same lines. We
next will refine Theorem \ref{Chthm} to obtain a Giambelli polynomial
representing any Schubert class in the even orthogonal case.

For the rest of this section, we assume that $N=2n+2$ is even, so that
$m=n+1-k$ and $K=2k>0$.  Fix a maximal isotropic subspace $L$ of $V$,
i.e.\ with $\dim(L)=n+1$.  Two maximal isotropic subspaces $E$ and $F$
of $V$ are said to be in the same family if $\dim(E\cap F)\equiv n+1
\,(\text{mod}\, 2)$.  The Schubert varieties in $\OG$ are defined
relative to an isotropic flag $F_\bull$, and their classes are
independent of this flag as long as $F_{n+1}$ is in the same family as
$L$.

A {\em typed $k$-strict partition} $\la$ consists of a $k$-strict
partition $(\la_1, \dots, \la_\ell)$ together with an integer
$\type(\la) \in \{0,1,2\}$, such that $\type(\la)>0$ if and only if
$\la_j=k$ for some index $j$.  Let $\wt{\cP}(k,n)$ denote the set of
all typed $k$-strict partitions whose Young diagrams are contained in
an $m\times (n+k)$ rectangle.  Notice that $\la_j=k<\la_{j-1}$ if and
only if $\ov{p}_j(\la) = n+2$.  For every $\la\in\wt{\cP}(k,n)$,
define the index function $p_j=p_j(\la)$ by
\[
 p_j(\lambda) =  \begin{cases}
      \ov{p}_j(\la)-1 & \text{if $\lambda_j=k < \lambda_{j-1}$ and
        $n+j+\type(\lambda)$ is even}, \\
      \ov{p}_j(\la) & \text{otherwise}.
   \end{cases}
\]
According to \cite{BKT1}, for every isotropic flag $F_\bull$ we have a
Schubert cell $X^{\circ}_{\la}=X^{\circ}_{\la}(F_\bull)$ in $\OG$, defined as the
locus of $\Sigma \in \OG$ such that $\dim(\Sigma \cap F_i) = 
\#\{j\ |\ p_j \leq i \}$
for each $i$. The Schubert variety $X_{\la}$ is
the Zariski closure of the Schubert cell $X^{\circ}_{\la}$. We let
$\ta_{\la}=[X_{\la}]$ denote the corresponding Schubert class in
$\HH^{2|\la|}(\OG,\Z)$; this class has a type which agrees with the
type of $\la$.

We say that a $k$-strict partition $\la$ has {\em positive type} if
$\la_i=k$ for some index $i$. If the (untyped) $k$-strict partition
$\la$ has positive type, we agree that $\ta_\la=[X_\la]$ and
$\ta'_\la=[X'_\la]$ denote the Schubert classes in
$\HH^{2|\la|}(\OG(m,2n+2))$ of type $1$ and $2$, respectively,
associated to $\la$. If $\la$ does not have positive type, then
$\ta_\la=[X_\la]$ denotes the associated Schubert class of type zero.
We then have that $[Y_\la]=\ta_\la+\ta'_\la$, if $\la$ has positive
type, while $[Y_\la]=\ta_\la$, otherwise.

The non-trivial automorphism of the Dynkin diagram of type
$\mathrm{D}_{n+1}$ gives rise to an involution $\iota$ of $\OG$, which
interchanges $X_\la$ and $X'_\la$. This in turn results in a weight
space decomposition
\begin{equation}
\label{decomp}
\HH^*(\OG,\Q) = \HH^*(\OG,\Q)_1 \oplus \HH^*(\OG,\Q)_{-1}.
\end{equation}

\begin{iprop}
\label{prop+1}
  The $\iota$-invariant subring of $\HH^*(\OG,\Q)$ is generated by the
  Chern classes of $\cQ$ and is spanned by the classes $[Y_\la]$, i.e.,
  \[
  \HH^*(\OG,\Q)_1 = \Q[c_1(Q), \dots, c_{n+k}(Q)]
  = \bigoplus_\la \Q \cdot [Y_\la] \,,
  \]
where the sum is over all $k$-strict partitions $\la$ contained in an 
$m\times (n+k)$ rectangle.
\end{iprop}

There are certain special Schubert varieties in $\OG(m,2n+2)$, defined
by a single Schubert condition, as the locus of $\Sigma\in\OG$ which
non-trivially intersect a given isotropic subspace or its orthogonal
complement.  The corresponding {\em special Schubert classes}
\begin{equation}
\label{spcls}
\ta_1,\ldots,\ta_{k-1},\ta_k,\ta'_k,\ta_{k+1},\ldots,\ta_{n+k}
\end{equation}
are indexed by the typed $k$-strict partitions with a single non-zero
part, and generate the cohomology ring $\HH^*(\OG,\Z)$.  We have
$\type(\tau_k)=1$, $\type(\tau'_k)=2$, and in this case equation
(\ref{ctoY}) becomes
\begin{equation}
\label{ctotau}
c_p(\cQ)=
\begin{cases}
\ta_p &\text{if $p< k$},\\
\ta_k+\ta_k' &\text{if $p=k$},\\
2\ta_p &\text{if $p> k$}.
\end{cases}
\end{equation}

Set $\ov{\OG} = \OG(n-k,2n+1) = \OG(m-1,N-1)$.  If the $k$-strict
partition $\la$ is contained in an $(n-k) \times (n+k)$ rectangle, let
$\s_\la$ denote the corresponding Schubert class in $\HH^*(\ov{\OG},
\Z)$. Both $\HH^*(\ov{\OG},\Q)$ and $\HH^*(\OG,\Q)$ are modules over
the ring $\Q[c]:=\Q[c_1,\dots,c_{n+k}]$, where the variables $c_p$ act as
multiplication with the Chern classes of the respective quotient
bundles on $\ov{\OG}$ and $\OG$. For any $k$-strict partition $\la$ we
let $\la + k$ denote the partition obtained by adding one copy of $k$
to $\la$ (and arranging the parts in decreasing order).

\begin{iprop}\label{prop-1}
  The linear map
  $\HH^*(\ov{\OG},\Q) \to \HH^*(\OG,\Q)_{-1}$ defined by $\s_\la
  \mapsto \ta_{\la+k} - \ta'_{\la+k}$ is an isomorphism of
  $\Q[c_1,\dots,c_{n+k}]$-modules. Moreover, we have 
\[
  \HH^*(\OG,\Q)_{-1} = \Q[c_1,\dots,c_{n+k}] \cdot (\ta_k -
  \ta'_k) = \bigoplus_{\la} \Q \cdot 
(\ta_{\la+k} - \ta'_{\la+k}) \,,
\]
where the sum is over all $k$-strict partitions $\la$ contained in
an $(m-1)\times (n+k)$ rectangle, and
\[
\ta_{\la+k} - \ta'_{\la+k} = 2^{-\ell_k(\la)} (\tau_k-\tau'_k) 
\,\wt{R}^{\la} \, c_\la,
\]
where $\wt{R}^\la$ is defined by equation {\em (\ref{Req})} with $K=2k+1$.
\end{iprop}

The statement in Proposition~\ref{prop-1} that the isomorphism of
vector spaces
\begin{equation}
\label{ogisom}
\s_\la \mapsto \ta_{\la+k} - \ta'_{\la+k}
\end{equation}
is also an isomorphism of $\Q[c]$-modules is an
apparently new relation between the cohomology of even and odd
orthogonal Grassmannians.  This result depends crucially on our
convention for assigning types to Schubert classes on $\OG$, which was
introduced in \cite{BKT1}. The convention was chosen in loc.\ cit.\
because it results in a relatively simple Pieri formula for products
with the special Schubert classes. Our proof that (\ref{ogisom}) gives
a $\Q[c]$-module homomorphism is based on this Pieri rule; it would be
interesting to expose a direct geometric argument.

Theorem \ref{Chthm} and Proposition \ref{prop-1} imply our main
result, a Giambelli formula which expresses any Schubert class
$\ta_\la$ in terms of the above special classes.  Let $R$ be any
finite monomial in the operators $R_{ij}$ which appears in the
expansion of the power series $R^\la$ in (\ref{Req}).  If
$\type(\la)=0$, then set $R \star c_{\la} = c_{R \,\la}$. Suppose that
$\type(\la)>0$, let $d = \ell_k(\la)+1$ be the index such that
$\la_d=k < \la_{d-1}$, and set $\wh{\al} =
(\al_1,\ldots,\al_{d-1},\al_{d+1},\ldots,\al_\ell)$ for any integer
sequence $\al$ of length $\ell$. If $R$ involves any factors $R_{ij}$
with $i=d$ or $j=d$, then let $R \star c_{\la} = \frac{1}{2}\,c_{R
\,\la}$. If $R$ has no such factors, then let
\[
R \star c_{\la} = \begin{cases}
\ta_k \,c_{\wh{R \,\la}} & \text{if  $\,\type(\la) = 1$}, \\
\ta'_k \, c_{\wh{R \,\la}} & \text{if  $\,\type(\la) = 2$}.
\end{cases}
\]

\begin{thm}[Classical Giambelli for $\OG$]
\label{mainthm}
For every $\la\in \wt{\cP}(k,n)$, we have 
$$\ta_\la =2^{-\ell_k(\la)}R^{\la} \star c_{\la}$$
in the cohomology ring of $\OG(n+1-k,2n+2)$. 
\end{thm}

For example, consider the ring $\HH^*(\OG(4,12))$ (where $k=2$) and
the partition $\la=(3,2,2)$.  Then the Schubert class for $\la$ of
type 2 is given by
\begin{align*}
\ta'_\la & = \frac{1}{2}
\frac{1-R_{12}}{1+R_{12}}(1-R_{13})(1-R_{23}) \star c_{322} \\
&= \frac{1}{2}(1-2R_{12}+2R_{12}^2 - 2 R_{12}^3)(1-R_{13}-R_{23}+R_{13}R_{23}) 
\star c_{322} \\
&= \ta_3\ta'_2(\ta_2+\ta_2') - \ta_4\ta'_2\ta_1 +\ta_6\ta_1 - \ta_3^2\ta_1
+\ta_4\ta_3 -\ta_7.
\end{align*}
We remark that in general, the Giambelli formula expresses the
Schubert class $\ta_\la$ as a polynomial in the special Schubert
classes (\ref{spcls}) with {\em integer} coefficients.

The small quantum cohomology ring $\QH(\OG)$ is a $q$-deformation of the 
cohomology ring $\HH^*(\OG,\Z)$ whose structure constants are defined
by the three point, genus zero Gromov-Witten invariants of $\OG$. Here 
$q$ denotes a formal variable if $k\geq 2$, or a pair $(q_1,q_2)$ of
formal variables if $k=1$. The quantum Pieri rule of \cite{BKT1} implies 
that the ring $\QH(\OG)$ is generated by the special Schubert classes. 
The following quantum Giambelli formula is the analogue of 
\cite[Thm.\ 2]{BKT3} in the even orthogonal case.

\begin{thm}[Quantum Giambelli for $\OG$]
\label{qgiam}
  For every $\la\in\wt{\cP}(k,n)$, we have
  \[
  \ta_\la = 2^{-\ell_k(\la)}R^{\la}\star c_{\la}
  \]
  in the quantum cohomology ring $\QH(\OG(n+1-k,2n+2))$.  In other
  words, the quantum Giambelli formula for $\OG$ is the same as the
  classical Giambelli formula.
\end{thm}

As in \cite{BKT2}, we will use raising operators to define a family of
polynomials $\{\Eta_\la\}$ indexed by typed $k$-strict partitions
whose algebra agrees with the Schubert calculus in the stable
cohomology ring of $\OG$. Let $x=(x_1,x_2,\ldots)$ be an infinite
sequence of variables and $y=(y_1,\ldots,y_k)$ be a finite set of $k$
variables.  We define the functions $q_r(x)$ and $e_r(y)$ by the
equations
\[
\prod_{i=1}^{\infty}\frac{1+x_it}{1-x_it} = \sum_{r=0}^{\infty}q_r(x)t^r
\ \ \ \mathrm{and} \ \ \ 
\prod_{j=1}^k(1+y_jt) = \sum_{r=0}^ke_r(y)t^r
\]
and set $\vartheta_r = \vartheta_r(x\,;y) = \sum_{i=0}^r
q_{r-i}(x)e_i(y)$ for each $r\geq 0$. (The $\vartheta_r$ will play the
role of the Chern classes $c_r(\cQ)$.) Define $\eta_r = \ti_r$ for
$r<k$, $\eta_r=\frac{1}{2}\ti_r$ for $r>k$, and set $\eta_k =
\frac{1}{2}\ti_k + \frac{1}{2}e_k(y)$ and $\eta'_k = \frac{1}{2}\ti_k
- \frac{1}{2}e_k(y) = \frac{1}{2}\sum_{i=0}^{k-1}q_{k-i}(x)e_i(y)$.
We call $B^{(k)}= \Z[\eta_1, \ldots, \eta_{k-1}, \eta_k, \eta'_k,
  \eta_{k+1}, \ldots]$ the ring of {\em eta polynomials}.  For any
typed $k$-strict partition $\la$, define the eta polynomial
\begin{equation}
\label{Heq}
H_\la = 2^{-\ell_k(\la)}\,R^\la \star \ti_\la. 
\end{equation}
The raising operator expression in (\ref{Heq}) is defined in the same
way as the analogous one in Theorem \ref{mainthm}, but using $\ti_r$
and $\eta_k,\eta'_k$ in place of $c_r$ and $\tau_k, \tau'_k$,
respectively.

\begin{thm}
\label{productthm}
The $\Eta_\la$, for $\la$ a typed $k$-strict partition, form a
$\Z$-basis of $B^{(k)}$. There is a surjective ring homomorphism
$B^{(k)} \to \HH^*(\OG(n+1-k,2n+2),\Z)$ such that $\Eta_\la$ is mapped to
$\tau_\la$, if $\la$ fits inside an $(n+1-k)\times (n+k)$ rectangle,
and to zero, otherwise.
\end{thm}

Furthermore, we prove that the eta polynomial $\Eta_\la(x\,;y)$ is
equal to the type D Schubert polynomial $\DS_{w_\la}(x, y)$ of Billey
and Haiman \cite{BH} indexed by the corresponding $k$-Grassmannian
element $w_\la$ of the Weyl group of type D.  A general theorem of
loc.\ cit.\ shows that $\DS_{w_\la}$ can be written as a sum 
of products of type D Stanley symmetric functions
and type A Schubert polynomials. Lam \cite{La} has shown that the type
D Stanley symmetric functions are positive integer linear combinations
of Schur $P$-functions, where the coefficients count Kra\'skiewicz-Lam
tableaux. Using these results, we can express $\Eta_\la(x\,;y)$ as an
explicit positive linear combination of products of Schur
$P$-functions and $S$-polynomials (Theorem \ref{bhthm}). A different 
expression for $\Eta_\la(x\,;y)$, which writes it as a sum of monomials
$2^{n(U)}(xy)^U$ over all `typed $k'$-bitableaux' $U$ of shape $\la$, 
is obtained in \cite{T3}.

This paper is organized as follows.  Sections \ref{classpieri} and
\ref{Wpieriproof} are concerned with the proof of Theorem \ref{Chthm}.
Propositions \ref{prop+1} and \ref{prop-1} and Theorem \ref{mainthm}
are proved in section \ref{classgiam}.  The quantum Giambelli formula
(Theorem \ref{qgiam}) is established in section \ref{qgog}, where we
also give a quantum version of Proposition~\ref{prop-1}.  Section
\ref{etasec} develops the theory of eta polynomials and contains the
proof of Theorem \ref{productthm}. In section \ref{BH} we
show that the eta polynomials are equal to certain Billey-Haiman
Schubert polynomials of type D, and give some applications. Finally,
the appendix contains a detailed study of the Schubert varieties 
in orthogonal Grassmannians which justifies our claims about
the spaces $Y_\la$, and corrects a related error in \cite{BKT1}.

This project was completed in part during stays at the Hausdorff
Research Institute for Mathematics in Bonn and the Forschungsinstitut
Oberwolfach in 2011; we thank both institutions for their hospitality
and stimulating environments. We are grateful to Vijay Ravikumar for
pointing out that the definition of the Schubert varieties in \cite[\S
  3.1]{BKT1} and an earlier version of this paper is mistaken.

\section{The Pieri rule for Chern classes}
\label{classpieri}

In this section we let the integer $N$ have arbitrary parity and work
in the cohomology ring $\HH^*(\OG(m,N),\Z)$. Recall that the integer
$K=N-2m$ is equal to $2k$ or $2k+1$. Let $c_p=c_p(\cQ)$ be the $p$-th
Chern class of the universal quotient bundle $\cQ$ over $\OG$ and
define $Y_\la\subset\OG$ as in the introduction. We will formulate a
Pieri rule for the cup products $c_p\cdot [Y_\la]$ in $\HH^*(\OG,\Z)$.

We identify each partition $\la$ with its Young diagram of
boxes. Given two Young diagrams $\mu$ and $\nu$ with $\mu\subset\nu$,
the skew diagram $\nu/\mu$ is called a horizontal (resp.\ vertical)
strip if it does not contain two boxes in the same column (resp.\
row). We say that the boxes $[r,c]$ and $[r',c']$ in row
$r$ (resp.\ $r'$) and column $c$ (resp.\ $c'$) of $\lambda$ are {\em
$K$-related} if $c\leq k < c'$ and $c+c' = K+1+r-r'$. This notion
also makes sense for boxes outside the Young diagram of $\la$.

For any two $k$-strict partitions $\lambda$ and $\mu$, we write
$\lambda \to \mu$ if $\mu$ may be obtained by removing a vertical
strip from the first $k$ columns of $\lambda$ and adding a horizontal
strip to the result, so that

\medskip
\noin (1) if one of the first $k$ columns of $\mu$ has the same number
of boxes as the same column of $\lambda$, then the bottom box of this
column is $K$-related to at most one box of $\mu \smallsetminus
\lambda$; and

\medskip
\noin
(2) if a column of $\mu$ has fewer boxes than the same column of
$\lambda$, then the removed boxes and the bottom box of $\mu$ in this
column must each be $K$-related to exactly one box of $\mu
\smallsetminus \lambda$, and these boxes of $\mu \smallsetminus
\lambda$ must all lie in the same row.

\medskip

Let $\A$ be the set of boxes of $\mu\ssm \la$ in columns $k+1$ and
higher which are not mentioned in (1) or (2), and define
$N(\lambda,\mu)$ to be the number of connected components of $\A$.
Here two boxes are connected if they share at least a vertex.

We say that a $k$-strict partition $\la$ has {\em positive type} if
$\la_i=K/2$ for some index $i$ (note that this can only happen if $K$ 
is even, so equal to $2k$). Given two $k$-strict partitions
$\la$ and $\mu$ with $\la\to\mu$, define
\[
\wh{N}(\la,\mu) = \begin{cases}
N(\lambda,\mu)+1 & \text{if $\la$ has
positive type and $\mu$ does not}, \\
N(\la,\mu) & \text{otherwise}.
\end{cases}
\]
For any $k$-strict partition $\la$ and any integer $p\geq 1$, the
multiplication rule
\begin{equation}
\label{whYpierieq}
c_p \cdot [Y_\lambda] = \sum_{\lambda \to \mu, \,
  |\mu|=|\lambda|+p} 2^{\wh{N}(\lambda,\mu)} \, [Y_\mu]
\end{equation}
holds in $\HH^*(\OG,\Z)$. Indeed, when $N$ is odd, (\ref{whYpierieq})
is equivalent to the Pieri rule for odd orthogonal Grassmannians from
\cite[Thm.\ 2.1]{BKT1}. When $N$ is even, the result follows easily
from the Pieri rule for even orthogonal Grassmannians \cite[Thm.\
3.1]{BKT1}, which is recalled in \S \ref{classgiam}.

For any $k$-strict partition $\la$ contained in an $m \times
(n+k)$ rectangle, define a class $W_\la$ in the cohomology ring of 
$\OG(m,N)$ by
\[
W_\la = 2^{-\ell_k(\la)}R^{\la} \, c_{\la}.
\]
The next result extends \cite[Eqn.\ (14)]{BKT2} to include
the even values of $N$.

\begin{thm}
\label{whWpieri}
We have
\begin{equation}
\label{whWpierieq}
c_p \cdot W_\lambda = \sum_{\lambda \to \mu, \,
  |\mu|=|\lambda|+p} 2^{\wh{N}(\lambda,\mu)} \, W_\mu
\end{equation}
in $\HH^*(\OG,\Z)$. In other words, the cohomology classes 
$[Y_\la]$ and $W_\la$ satisfy the same Pieri rule for products with 
the Chern classes of $\cQ$.
\end{thm}

Observe that Theorem 1 follows easily from
Theorem \ref{whWpieri}.  To see this, write $\mu \succ \la$ if
$\mu$ strictly dominates $\lambda$, i.e., $\mu \neq \lambda$ and
$\mu_1 + \dots + \mu_i \geq \lambda_1 + \dots + \lambda_i$ for each $i
\geq 1$.  It follows from (\ref{whYpierieq}) and (\ref{whWpierieq}) that
\[
2^{\ell_k(\la)} W_\la + \sum_{\mu \succ \la} a_{\la\mu}\,W_\mu =
c_{\la_1}\cdots c_{\la_\ell} = 2^{\ell_k(\la)}[Y_\la] + \sum_{\mu
\succ \la} a_{\la\mu}\,[Y_\mu]
\]
for some constants $a_{\la\mu}\in\Z$. Using this and induction on
$\la$, we deduce that $W_\la = [Y_\la]$, for each $k$-strict
partition $\la$. We will prove Theorem \ref{whWpieri} (and hence
also Theorem 1) in the following section.

\section{Proof of Theorem \ref{whWpieri}}
\label{Wpieriproof}

Our proof of Theorem \ref{whWpieri} is almost identical to the proof
of \cite[Eqn.\ (14)]{BKT2}. We will not repeat the arguments of loc.\
cit.\ here, but will give an overview of the proof, pointing out where
it needs to be modified to include the even case $K=2k$.

\subsection{}
If $\la$ is any sequence of (possibly negative) integers, we say that
$\la$ has length $\ell$ if $\la_i=0$ for all $i>\ell$ and $\ell\geq 0$
is the smallest number with this property.  All integer sequences in
this paper have finite length.  In analogy with Young diagrams of
partitions, we will say that a pair $[i,j]$ is a {\em box} of the
integer sequence $\lambda$ if $i \geq 1$ and $1 \leq j \leq
\lambda_i$.  A {\em composition} $\al =
(\al_1,\al_2,\dots,\al_r,\ldots)$ is a sequence of integers from the
set $\N=\{0,1,2,\ldots\}$; we let $|\al| = \sum \al_i$.

Let $\Delta = \{(i,j) \in \N \times \N \mid 1\leq i \leq j \}$ and
define a partial order on $\Delta$ by agreeing that $(i',j')\leq
(i,j)$ if $i'\leq i$ and $j'\leq j$.  We call a finite subset $D$ of
$\Delta$ a {\em valid set of pairs} if $(i,j)\in D$ implies
$(i',j')\in D$ for all $(i',j')\in \Delta$ with $(i',j') \leq
(i,j)$. An {\em outer corner} of a valid set of pairs $D$ is a pair
$(i,j) \in \Delta \ssm D$ such that $D \cup (i,j)$ is also a valid set
of pairs.

\begin{defn}
\label{recursedef}
For any valid set of pairs $D$, we define the raising operator
\[
R^D = \prod_{i<j}(1-R_{ij})\prod_{i<j\, :\, (i,j)\in D}(1+R_{ij})^{-1}.
\]
For any finite monomial $R$ in the $R_{ij}$'s which appears in the 
expansion of $R^D$, and any integer sequence $\la$, we let 
$R\, c_\la = c_{R\la}$. We define the element $T(D,\la)$ in 
$\HH^*(\OG(m,N),\Z)$ by the formula 
\[
T(D,\la)= 2^{-\#\{i\,| \,(i,i)\in D\}} \, R^D c_{\la}.
\]
\end{defn}

It follows from \cite[Thms.\ 2.2 and 3.2]{BKT1} that the Chern
classes $c_r$ satisfy the relations
\begin{equation}
\label{kpresrels}
\frac{1-R_{12}}{1+R_{12}}\, c_{(r,r)} = 
c_r^2 + 2\sum_{i=1}^r(-1)^i c_{r+i}c_{r-i}= 0
\ \ \ \text{for} \  r > k.
\end{equation}
The next three lemmas are proved using the relations (\ref{kpresrels}) 
in the same way as their counterparts in \cite[Lemmas 1.2--1.4]{BKT2}.

\begin{lemma}\label{commuteA}
Let $\lambda=(\lambda_1,\ldots,\lambda_{j-1})$ and
$\mu=(\mu_{j+2},\ldots,\mu_\ell)$ be integer vectors.  Assume that
$(j,j+1)\notin D$ and that for each $h<j$, $(h,j)\in D$ if and only if
$(h,j+1)\in D$.  Then for any integers $r$ and $s$ we have
\[ T(D,(\lambda,r,s,\mu)) = - T(D,(\lambda,s-1,r+1,\mu)) \,.  \]
In particular, $T(D,(\lambda, r, r+1, \mu))=0$.
\end{lemma}

\begin{lemma}\label{commuteC}
Let $\lambda=(\lambda_1,\ldots,\lambda_{j-1})$ and
$\mu=(\mu_{j+2},\ldots,\mu_\ell)$ be integer vectors, assume
$(j,j+1)\in D$, and that for each $h>j+1$, $(j,h)\in D$ if and only if
$(j+1,h)\in D$.  If $r,s\in \Z$ are such that $r+s > 2k$, then we have
\[ T(D,(\lambda,r,s,\mu)) = - T(D,(\lambda,s,r,\mu)) \,.  \]
In particular, $T(D,(\lambda,r,r,\mu)) = 0$ for any $r>k$.
\end{lemma}

\begin{lemma}
\label{easylm}
If $(i,j)\notin D$ and $D\cup (i,j)$ is a valid set of pairs, then
\[
T(D, \lambda) = T(D\cup (i,j), \lambda) +
T(D\cup (i,j), R_{ij}\lambda).
\]
\end{lemma}

\subsection{}
\label{initnot}
Throughout the rest of this section we fix $K$, $p>0$, and the
$k$-strict partition $\la$ of length $\ell$.  We will work in the ring
$\HH^*(\OG,\Z)$, where $\OG = \OG(m,2m+K)$ for a sufficiently large
integer $m$. Define a valid set of pairs $\cC=\cC(\la)$ by
\[
\cC(\la) = \{(i,j)\in\Delta\ |\ \la_i+\la_j \geq K+j-i, \ \la_i > k, 
\ \, \text{and} \ \, j \leq \ell\}.
\]
Notice that $W_\la = T(\cC,\la)$. For any $d\geq \ell$ define the
raising operator $R^{\la}_d$ by
\[
R_d^{\la} = \prod_{1\leq i<j\leq d}(1-R_{ij})\,
\prod_{i<j\, :\, (i,j)\in\cC} (1+R_{ij})^{-1}.
\]
We compute that
\[
c_p\cdot W_\la = c_p\cdot 2^{-\ell_k(\la)}R_{\ell}^{\la}\, c_{\la} =
2^{-\ell_k(\la)} R_{\ell+1}^\la\cdot
\prod_{i=1}^\ell(1-R_{i,\ell+1})^{-1} \, c_{\la,p}
\]
\[
= 2^{-\ell_k(\la)}
R^\la_{\ell+1}\cdot\prod_{i=1}^\ell(1+R_{i,\ell+1} +
R_{i,\ell+1}^2 + \cdots)\,c_{\la,p}
\]
and therefore
\begin{equation}\label{initeq}
  c_p\cdot T(\cC,\la) = \sum_{\nu\in \cN}T(\cC,\nu),
\end{equation}
where $\cN=\cN(\la,p)$ is the set of all compositions $\nu \geq \la$
such that $|\nu| = |\la|+p$ and $\nu_j =0$ for $j > \ell+1$. 

\subsection{}\label{initialdefs}

We will prove that the right hand side of equation (\ref{initeq}) is
equal to the right hand side of the Pieri rule (\ref{whWpierieq}),
thus proving Theorem \ref{whWpieri}.  For the rest of this section we
set $m = \ell_k(\la)+1$, i.e.\ $m$ is minimal such that $\la_m \leq
k$.  We call $m$ the {\em middle row\/} of $\la$.

\begin{defn}
  A {\em valid 4-tuple} of {\em level} $h$ is a 4-tuple $\psi =
  (D,\mu,S,h)$, such that $h$ is an integer with $0 \leq h \leq
  \ell+1$, $D$ is a valid set of pairs containing $\cC$, all pairs
  $(i,j)$ in $D$ satisfy $i\leq m$ and $j \leq \ell+1$, $S$ is a
  subset of $D \ssm \cC$, and $\mu$ is an integer sequence of length
  at most $\ell+1$.  The evaluation of $\psi$ is defined by $\ev(\psi)
  = T(D,\mu) \in \HH^*(\OG,\Z)$.
\end{defn}

In the following we set $\mu_0 = \infty$ whenever $\mu$ is an integer
sequence.

\begin{defn}
  For any $y \in \Z$ we let $r(y)$ denote the largest integer such
  that $r(y) \leq \ell+1$ and $\la_{r(y)-1} \geq K+r(y)-y$.
\end{defn}

We  have the relation
$\cC = \{ (i,j) \in \Delta \mid
   j < r(i+\la_i+1) \ \, \text{and} \ \, \la_i > k \}$.
Notice also that $r(m+k)=m$ while
\begin{equation}
\label{rmk}
r(m+k+1)=
\begin{cases}
m+1 & \text{if $\la_m=K/2$}, \\
m & \text{otherwise}. 
\end{cases}
\end{equation}

\begin{defn}\label{efg}
  Let $h\in \N$ satisfy $1\leq h \leq m$ and let $\mu$ be an integer
  sequence.\smallskip

  \noindent
  (a) We define $b_h = r(h+\la_h+1)$ and $g_h = b_{h-1}$.  By
  convention we set $g_1 = \ell+1$.  \smallskip

  \noindent
  (b) Set $R(\mu) = \{ [i,c] \in \mu\ssm\la \mid c>k \text{ and }
  \mu_{r(i+c)} < K+r(i+c)-i-c \}$.
  \smallskip

  \noindent
  (c) Assume that $h \geq 2$ and $\mu_h \geq \la_{h-1}$.  If
  $[h,\la_{h-1}] \in R(\mu)$ then set $e_h(\mu) = \la_{h-1}$.
  Otherwise, if $h<m$ (respectively, if $h=m$) choose $e_h(\mu) >
  \la_h$ (respectively, $e_h(\mu) \geq K/2$) minimal such that $[h,c]
  \not\in R(\mu)$ for $e_h(\mu) \leq c \leq \la_{h-1}$.  Finally, set
  $f_h(\mu) = r(h+e_h(\mu))$.
\end{defn}

If $\mu$ is a $k$-strict partition such that $\la \to \mu$, then the
set $\A$ from \S \ref{classpieri} consists of the boxes of
$\mu\ssm\la$ in columns $k+1$ and higher which are not in $R(\mu)$.
For a general sequence $\mu$, when $[h,\la_{h-1}] \notin R(\mu)$, the
integer $e_h(\mu)$ is the least such that $e_h(\mu) \geq K/2$,
$(e_h(\mu),\la_h)$ is a $k$-strict partition, and $[h,c] \not\in
R(\mu)$ for $e_h(\mu) \leq c \leq \la_{h-1}$.

If we are given a fixed valid 4-tuple $(D,\mu,S,h)$ with $1\leq h \leq
m$, we will use the shorthand notation $b = b_h$, $g=g_h$, $R =
R(\mu)$, $e = e_h(\mu)$, and $f = f_h(\mu)$.

\begin{defn}\label{wvdef}
  Let $(i,j)\in \Delta$ be arbitrary.
  We define two conditions $\W(i,j)$ and $\X$ on a valid $4$-tuple
  $(D,\mu,S,h)$ as follows.
  \[
  \W(i,j)\ :\ \mu_i+\mu_j \geq K + j-i \ \, \text{and} \ \, \mu_i > k.
  \]
  Condition $\X$ is true if and only if $(h,h) \in D$ and
  \[
  \mu_h \geq \mu_{h-1} \  {\mathrm{or}}  \ \mu_h > \la_{h-1} \
  {\mathrm{or}} \ (\mu_h = \la_{h-1} \ {\mathrm{and}}  \
  (h,\f)\notin S) \,.
  \]
\end{defn}

\subsection{}\label{ss:subrule}

 The following {\em substitution rule} will be applied iteratively to
rewrite the right hand side of (\ref{initeq}). Both this rule and the
algorithm which follows it are identical to the one in \cite[\S
3.3]{BKT2}, but we recall them here for the sake of exposition. 

\medskip

\begin{center}
{\bf \underline{Substitution Rule}}
\end{center}

\medskip
\medskip

Let $(D,\mu,S,h)$ be a valid 4-tuple of level $h\geq 1$.  Assume first
that $(h,h)\notin D$.  If

\medskip
\begin{center}
{\bf (i)} there is an outer corner $(i,h)$ of $D$ with $i\leq m$
such that $\W(i,h)$ holds
\end{center}

\medskip
\noin
then REPLACE $(D,\mu,S,h)$ with
\[
(D \cup (i,h),\mu,S,h) \ \ \mathrm{and} \ \ (D \cup (i,h),R_{ih}\mu,
S\cup (i,h), h).
\]

\noin
Otherwise, if

\medskip
\begin{center}
{\bf (ii)} $D$ has no outer corner in column $h$ and $\mu_h > \la_{h-1}$,
\end{center}

\medskip
\noin
then STOP.

\medskip
Assume now that $(h,h)\in D$.  If

\medskip
\begin{center}
{\bf (iii)} there is an outer corner $(h,j)$ of $D$ with $j \leq \ell+1$
such that $\W(h,j)$ holds,
\end{center}

\medskip
\noin
then REPLACE $(D,\mu,S,h)$ with
\[
\begin{cases}
(D \cup (h,j),\mu,S,h)\ \ \mathrm{and} \ \ (D\cup (h,j),R_{hj}\mu, S\cup
(h,j),h) & \mathrm{if} \ \mu_j \leq \mu_{j-1}, \\
(D\cup (h,j),R_{hj}\mu, S\cup(h,j),h) & \mathrm{if} \
\mu_j > \mu_{j-1}.
\end{cases}
\]
Otherwise, if

\medskip
\begin{center}
  {\bf (iv)} $\W(h,\g)$ or $\X$ holds, and $D$ has an
  outer corner $(i,g)$ with $i \leq h$,
\end{center}

\medskip
\noin
then REPLACE $(D,\mu,S,h)$ with
\[
(D \cup (i,\g),\mu,S,h) \ \ \mathrm{and} \ \ (D \cup (i,\g),R_{i\g}\mu,
S\cup (i,\g), h).
\]
Otherwise, if

\medskip
\begin{center}
{\bf (v)}  $\X$ holds,
\end{center}

\medskip
\noin
then STOP.

\medskip

If  none of the above conditions hold, REPLACE
$(D,\mu,S,h)$ with $(D,\mu,S,h-1)$.

\subsection{}

Define the set $\Psi = \{(\cC,\nu,\emptyset,\ell+1) \mid \nu\in
\cN(\la,p)\}$, so that $\sum_{\psi\in\Psi}\ev(\psi)$ agrees with the
right hand side of (\ref{initeq}). Consider the following algorithm
which will change $\Psi$ by replacing some $4$-tuples with one or two
new valid $4$-tuples.  The algorithm applies the Substitution Rule to
each element $(D,\mu,S,h)$ of level $h \geq 1$.  If the substitution
rule results in a REPLACE statement, then the set is changed by
replacing $(D,\mu,S,h)$ by one or two new 4-tuples accordingly;
otherwise the substitution rule results in a STOP statement, and the
$4$-tuple $(D,\mu,S,h)$ is left untouched.  These substitutions are
iterated until no further elements can be REPLACED.

Suppose that the $4$-tuple $\psi=(D,\mu,S,h)$ occurs in the
algorithm. If $\psi$ is replaced by two 4-tuples $\psi_1$ and
$\psi_2$, then it follows from Lemma \ref{easylm} that $\ev(\psi) =
\ev(\psi_1) + \ev(\psi_2)$.  Moreover, if $\psi$ meets {\bf (iii)} and
is replaced by the single 4-tuple $\psi'=(D\cup (h,j),R_{hj}\mu,
S\cup(h,j),h)$, then one can show that $\mu_j=\mu_{j-1}+1$ and $D\cup
(h,j)$ has no outer corner in column $j$, so Lemmas \ref{commuteA} and
\ref{easylm} imply that $\ev(\psi)=\ev(\psi')$.

When the algorithm terminates, let $\Psi_0$ (respectively $\Psi_1$)
denote the collection of all $4$-tuples $(D,\mu,S,h)$ in the final set
such that $h=0$ (respectively $h>0$). We deduce from the above
analysis that
\[
\sum_{\nu\in \cN}T(\cC,\nu) = \sum_{\psi\in\Psi_0}\ev(\psi) +
\sum_{\psi\in\Psi_1}\ev(\psi).
\]

\begin{claim}
\label{claim1}
For each 4-tuple $\psi = (D,\mu,S,0)$ in $\Psi_0$ with $\mu_{\ell+1}
\geq 0$, $\mu$ is a $k$-strict partition with $\lambda \to \mu$, 
\[
D = \cC_{\ell+1}(\mu) := \{(i,j)\in\Delta\ |\ \mu_i+\mu_j 
\geq K+j-i, \ \, \mu_i>k, 
\ \, \text{and} \ \, j \leq \ell+1\}
\]
is uniquely determined by $\mu$, and $\ev(\psi) =
T(\cC(\mu),\mu)$. Furthermore, for each such partition
$\mu$, there are exactly $2^{\wh{N}(\lambda,\mu)}$ such 4-tuples
$\psi$, in accordance with the Pieri rule (\ref{whWpierieq}). 
\end{claim}

Given a valid 4-tuple $\psi = (D,\mu,S,h)$ with $h \geq 2$ and $\mu_h
\geq \la_{h-1}$, we define a new 4-tuple $\iota\psi$ as follows.  If
$(h,h) \notin D$, then set $\iota\psi = (D,\wt\mu,S,h)$, where the
composition $\wt\mu$ is defined by $\wt\mu_{h-1} = \mu_h-1$,
$\wt\mu_h = \mu_{h-1}+1$, and $\wt\mu_t = \mu_t$ for $t \notin
\{h-1,h\}$.  If $(h,h) \in D$ and $\mu_{h-1} = \mu_h$, then set
$\iota\psi = \psi$.  Assume that $(h,h) \in D$ and $\mu_{h-1} \neq
\mu_h$.  Let $\varpi$ be the involution of $\Delta$ that exchanges
$(h-1,g)$ with $(h,f)$, and fixes all other pairs.  Then set
$\iota\psi = (D,\wt\mu,\wt S,h)$, where $\wt S = \varpi(S)$,
and $\wt\mu$ is the composition obtained from $\mu$ by switching
the parts $\mu_{h-1}$ and $\mu_h$.

\begin{claim}
\label{claim2}
The above map $\iota$ restricts to an involution $\Psi_1 \to\Psi_1$
such that $\ev(\psi) + \ev(\iota(\psi)) = 0$, for every
$\psi\in\Psi_1$.
\end{claim}

We remark that the 4-tuples $\psi\in\Psi_0$ with $\mu_{\ell+1} <0$
evaluate to zero trivially, by Definition \ref{recursedef}; the two
claims therefore suffice to prove Theorem \ref{whWpieri}. The proofs
of these claims are nearly identical to the ones in \cite[\S 4]{BKT2},
replacing the value $2k+1$ by $K$ throughout. The following explicit
construction of the sets $S$ in the 4-tuples which appear in $\Psi_0$
accounts for the multiplicities $\wh{N}(\la,\mu)$ in Claim
\ref{claim1}.

Fix an arbitrary $k$-strict partition $\mu$ such that $\la \to \mu$
and $|\mu|=|\la|+p$. Recall the set $\A$ of \S \ref{classpieri}, and
define a new set $\wh{\A}$ by
\[
\wh{\A} = \begin{cases}
\A\cup\{[m,k]\} & \text{if $\la_m = K/2 < \mu_m$}, \\
\A & \text{otherwise}.
\end{cases}
\]
A {\em component} means an (edge or vertex) connected component of the
set $\wh{\A}$.  We say that a box $B$ of $\wh{\A}$ is {\em
distinguished} if the box directly to the left of $B$ does not lie in
$\wh{\A}$.  We say that $B$ is {\em optional} if it is the rightmost
distinguished box in its component.  Using (\ref{rmk}), we deduce that
$\wh{N}(\la,\mu)$ is equal to the number of optional distinguished
boxes in $\wh{\A}$.

To each distinguished box $B = [i,c]$ we associate the pair $(i,j) =
(i,r(i+c))$. The inequality $\la_{i-1} \geq K+i-(i+c)$ implies that
$i\leq j$, so $(i,j) \in \Delta$. Let $E$ (respectively $F$) be the
set of pairs associated to optional (respectively non-optional)
distinguished boxes.  We furthermore let $G$ be the set of all pairs
$(i,j) \in \Delta$ for which some box in row $i$ of $\mu\ssm\la$ is
$K$-related to a box in row $j$ of $\la\ssm\mu$.

Suppose that $(\cC_{\ell+1}(\mu),\mu,S,0) \in \Psi_0$ and $(i,j) \in S
\ssm G$. Then one can show that $(i,j)$ is the pair associated to a
distinguished box of $\wh{\A}$. For the last statement, observe that
if $i=j=m$, then $\la_m \leq K/2 < \mu_m$ and $(m,m)$ is associated to
the distinguished box $[m,k+1]\in\wh{\A}$ (respectively, $[m,k]\in
\wh{\A}$) if $\la_m<K/2$ (respectively, $\la_m=K/2$). To every subset
$E'$ of $E$ we associate the set of pairs $S(E') := E'\cup F\cup G$.
This is a disjoint union, and there are exactly $2^{\wh{N}(\la,\mu)}$
sets of this form.  Let $S \subset \Delta$ be any subset.  Then one
may prove as in \cite[\S 4]{BKT2} that $(\cC_{\ell+1}(\mu),\mu,S,0)
\in \Psi_0$ if and only if $S = S(E')$ for some subset $E' \subset E$.

\section{The classical Giambelli formula}
\label{classgiam}

In this section we prove Propositions \ref{prop+1} and \ref{prop-1}
and Theorem \ref{mainthm} of the introduction. Throughout the section
we will be in the even orthogonal case where $K=2k$. We assume that
$k>0$, however all of our results about the cohomology of
$\OG(n+1-k,2n+2)$ also hold when $k=0$, provided that in the latter
case $\OG$ parametrizes both families of maximal isotropic subspaces,
and hence is a disjoint union of two irreducible components.

\subsection{}
The Schubert classes $\ta_\la$ for all typed $k$-strict partitions
$\la$ whose diagrams are contained in an $m\times (n+k)$ rectangle
form a $\Z$-basis of $\HH^*(\OG(m,N),\Z)$. For typed $k$-strict
partitions $\la$ and $\mu$, we write $\la\to\mu$ if the underlying
$k$-strict partitions satisfy $\la\to\mu$ (with $K=2k$) and
furthermore $\type(\la)+\type(\mu)\neq 3$.  According to \cite[Thm.\
3.1]{BKT1}, the following Pieri rule holds in $\HH^*(\OG,\Z)$. For any
typed $k$-strict partition $\la$ and any integer $p\geq 1$, we have
\begin{equation}
\label{bktpieri}
 c_p \cdot \tau_\lambda = \sum_{\substack{\lambda \to \mu \\
      |\mu|=|\lambda|+p}} 2^{N(\lambda,\mu)} \, \tau_\mu.
\end{equation}

\subsection{}
\label{mainpfs}
In the sequel we will have to work both with $k$-strict and with typed
$k$-strict partitions. We will therefore adopt the following
conventions for use with Schubert classes and (later) representing
polynomials indexed by such objects. If the $k$-strict partition
$\la$ has positive type, we agree that $\ta_\la$ and $\ta'_\la$ denote
the Schubert classes in $\HH^*(\OG,\Z)$ of type $1$ and $2$,
respectively, associated to $\la$. If $\la$ does not have positive
type, then $\ta_\la$ denotes the associated Schubert class of type
zero.

For any $k$-strict partition $\la$, we define a cohomology class
$\wh{\ta}_\la\in\HH^{2|\la|}(\OG,\Z)$ by the equations
\begin{equation}
\label{htadef}
\wh{\ta}_\la = \begin{cases}
\ta_\la + \ta'_\la & \text{if $\la$ has positive type}, \\
\ \quad \ta_\la & \text{otherwise}.
\end{cases}
\end{equation}
Let $\Q[c]$ denote the $\Q$-subalgebra of $\HH^*(\OG(m,N),\Q)$
generated by the Chern classes $c_p$ for all $p\geq 1$.

Recall from the introduction that for each $k$-strict partition $\la$
contained in an $m\times (n+k)$ rectangle, we have a Zariski closed
subset $Y_\la$ of the even orthogonal Grassmannian $\OG(n+1-k,2n+2)$,
defined by the equations (\ref{Yeq}). See the appendix for a proof
that if $\la$ has type zero, then $Y_\la$ is the Schubert variety
$X_\la$ in $\OG$, and otherwise $Y_\la$ has two irreducible
components, which are the Schubert varieties $X_\la$ and $X'_\la$. We
deduce that the cohomology class $[Y_\la]$ in $\HH^{2|\la|}(\OG,\Z)$
is equal to $\wh{\tau}_\la$, for each such $\la$.  The Pieri rule
(\ref{bktpieri}) therefore implies that the classes $[Y_\la]$ satisfy
the rule (\ref{whYpierieq}).  Proposition \ref{prop+1} follows
directly from these observations.

We next prove Proposition \ref{prop-1}. For any $k$-strict partition
$\la$, define a class $\wt{\ta}_\la$ in $\HH^{2|\la|}(\OG,\Z)$ by the
equations
\begin{equation}
\label{wtadef}
\wt{\ta}_\la = \begin{cases}
\ta_\la - \ta'_\la & \text{if $\la$ has positive type}, \\
\ \quad 0 & \text{otherwise}.
\end{cases}
\end{equation}
The Pieri rule (\ref{bktpieri}) implies that for any $p\geq 1$, we have
\begin{equation}
\label{bkwtpieri}
c_p \cdot \wt{\ta}_\lambda = \begin{cases}
\sum_{\lambda \to \mu, \,
  |\mu|=|\lambda|+p} 2^{N(\lambda,\mu)} \, \wt{\ta}_\mu &
\text{if $\la$ has positive type}, \\
\qquad \qquad \ 0 & \text{otherwise}.
\end{cases}
\end{equation}
The positivity of the coefficients in this formula is a consequence of
the type convention for Schubert classes introduced in \cite{BKT1}.

Consider the $\Q$-linear map
\[
\psi: \HH^*(\OG(m-1,N-1),\Q) \to \HH^*(\OG(m,N),\Q)_{-1}
\]
defined by $\psi(\sigma_\la) = \wt\tau_{\la+k}$ for each $k$-strict
partition $\la$ contained in an $(m-1)\times(n+k)$ rectangle.  By
comparing (\ref{bkwtpieri}) with (\ref{whYpierieq}) for $K=2k+1$, we
obtain $\psi(c_p \cdot \sigma_\la) = c_p \cdot \psi(\sigma_\la)$.  We
deduce that $\psi$ is an isomorphism of $\Q[c]$-modules, and that
\[
  \HH^*(\OG,\Q)_{-1} = \bigoplus_{\la} \Q \cdot 
  (\ta_{\la+k}-\ta'_{\la+k}) = \Q[c_1,\dots,c_{n+k}] \cdot (\ta_k - \ta'_k).
\]

Next, use Theorem \ref{Chthm} to expand $\sigma_\la\in
\HH^*(\OG(m-1,N-1))$ as a Giambelli polynomial in the Chern classes of
$\cQ$. By mapping this equation to $\HH^*(\OG(m,N),\Q)_{-1}$ via
$\psi$ and using the fact that $\psi(1)= \ta_k - \ta'_k$, it follows
that
\begin{equation}
\label{la+k}
\ta_{\la+k}-\ta'_{\la+k} = 2^{-\ell_k(\la)} (\tau_k-\tau'_k) 
\,\wt{R}^{\la} \, c_\la,
\end{equation}
where  $\wt{R}^\la$ denotes the operator defined by
formula (\ref{Req}) with $K=2k+1$.
This completes the proof of Proposition \ref{prop-1}.

Let $\la$ be a $k$-strict partition contained in an $m \times (n+k)$
rectangle, and let $R$ be a finite monomial in the operators $R_{ij}$
that occurs in the expansion of $R^\la$.  If $\la$ does not have
positive type, then set $R \diamond c_\la = 0$.  If $\la$ has positive
type, then set $d = \ell_k(\la)+1$, so that $\la_d = k < \la_{d-1}$.
If $R$ contains any operator $R_{ij}$ for which $i$ or $j$ is equal to
$d$, then set $R \diamond c_\la = 0$.  Otherwise define $R \diamond
c_\la = (\tau_k - \tau'_k)\, c_{\wh{R \la}}$, where $\wh \alpha =
(\al_1,\dots,\al_{d-1},\al_{d+1},\dots)$.  Equation (\ref{la+k}) is
then equivalent to the identity
\begin{equation}
\label{wtprop_eqn}
\wt \tau_\la = 2^{-\ell_k(\la)} R^\la \diamond c_\la \,.
\end{equation}

Now observe that for any typed $k$-strict partition $\la \in \wt
\cP(k,n)$ and monomial $R$ we have
\begin{equation}
\label{eqn:Rcases}
R \star c_\la = \begin{cases}
  R\, c_\la & \text{if $\type(\la)=0$,} \\
  \frac{1}{2} R\, c_\la + \frac{1}{2} R \diamond c_\la & \text{if
    $\type(\la)=1$,} \\
  \frac{1}{2} R\, c_\la - \frac{1}{2} R \diamond c_\la & \text{if
    $\type(\la)=2$.}
  \end{cases}
\end{equation}
Theorem~\ref{mainthm} follows by combining
Theorem~\ref{Chthm}, equations (\ref{wtprop_eqn}) and
(\ref{eqn:Rcases}), and the identity
\begin{equation}
\label{whwtta}
\ta_\la = \begin{cases}
\wh{\ta}_\la & \text{if $\type(\la)=0$}, \\
(\wh{\ta}_\la + \wt{\ta}_\la)/2 & \text{if $\type(\la)=1$}, \\
(\wh{\ta}_\la - \wt{\ta}_\la)/2 & \text{if $\type(\la)=2$}.
\end{cases}
\end{equation}
Note that we have ignored the type of $\la$ in the expressions
$R\,c_\la$, $R \diamond c_\la$, $\wh{\ta}_\la$, $\wt{\ta}_\la$ which
appear on the right hand sides of equations (\ref{eqn:Rcases}) and
(\ref{whwtta}).

\section{The quantum Giambelli formula}
\label{qgog}

\subsection{Quantum cohomology of orthogonal Grassmannians}
\label{S:qpieri}

As in the introduction, we consider the orthogonal Grassmannian 
$\OG = \OG(m,N)$ with $N=2m+K$, where $K \geq 2$. When $K \neq 2$, the
quantum cohomology ring $\QH(\OG)$ is a $\Z[q]$-algebra which is
isomorphic to $\HH^*(\OG) \otimes_\Z \Z[q]$ as a module over $\Z[q]$.
The degree of the formal variable $q$ is $n+k$.  When $K=2$,
$\QH(\OG)$ is a $\Z[q_1,q_2]$-algebra, we have $\QH(\OG) \cong
\HH^*(\OG,\Z) \otimes_\Z \Z[q_1,q_2]$ as a $\Z[q_1,q_2]$-module, and
$\deg(q_1) = \deg(q_2) = n+1$. In both cases, we set 
$\QH(\OG,\Q):=\QH(\OG)\otimes_\Z\Q$. 

We first correct an error in the definition of $\wt{\nu}$ given
in \cite[\S 3.5]{BKT1}, which appears in the quantum Pieri rule
\cite[Thm.\ 3.4]{BKT1}.  In the first paragraph of \cite[\S
3.5]{BKT1}, $\wt\nu$ should be obtained from $\nu$ by removing the
first row of $\nu$ as well as $n+k-\nu_1$ boxes in the first column,
i.e. $\wt\nu=(\nu_2, \nu_3, \dots, \nu_r)$, where $r=\nu_1-2k+2$.

We next recall some further definitions from \cite{BKT1}, using a
more uniform notation. Let $P(m,N)$ denote the set of all $k$-strict
partitions contained in an $m \times (n+k)$ rectangle.  For $\la \in
P(m,N)$ we let $\la^* = (\la_2,\la_3,\dots)$ be the partition obtained
by removing the first row of $\la$.  Let $P'(m,N)$ be the set of
$k$-strict partitions $\nu$ contained in an $(m+1)\times (n+k)$
rectangle for which $\ell(\nu)=m+1$, $\nu_1 \geq K-1$, and the number
of boxes in the second column of $\nu$ is at most $\nu_1-K+2$.  For
each element $\nu \in P'(m,N)$, we let $\wt\nu \in P(m,N)$ denote the
partition obtained by removing the first row of $\nu$ as well as
$n+k-\nu_1$ boxes from the first column.  That is,
\[
\wt\nu = (\nu_2, \nu_3, \dots, \nu_r)\, , 
\text{\ \ where $r = \nu_1-K+2$.}
\]

For each $k$-strict partition $\la \in P(m,N)$, we define
$Y_\la\subset \OG$ as in the introduction and let $[Y_\la]$ denote the
corresponding class in $\QH(\OG)$. For odd integers $N$, we set
$\sigma_\la=[Y_\la]$, while for even $N$, we let $\wh\tau_\la$ and
$\wt\tau_\la$ denote the classes in $\QH(\OG)$ defined by the
equations (\ref{htadef}) and (\ref{wtadef}). We proceed to describe
the products of these classes with the Chern classes $c_p = c_p(\cQ)$.

Assume first that $K \geq 3$. The following quantum Pieri rule for
Chern classes is a direct consequence of \cite[Thms.\
2.4 and 3.4]{BKT1}. For any $\la \in P(m,N)$ and integer $p \in
[1,n+k]$, we have
\begin{equation}\label{eqn:qchernpieri}
  c_p \cdot [Y_\la] = 
  \sum_{\la \to \mu} 2^{\wh N(\la,\mu)}\, [Y_\mu] +
  \sum_{\la \to \nu} 2^{\wh N(\la,\nu)}\, [Y_{\wt\nu}]\, q +
  \sum_{\la^* \to \rho} 2^{\wh N(\la^*,\rho)}\, [Y_{\rho^*}]\, q^2\,.
\end{equation}
Here the first sum is classical, the second sum is over $\nu \in
P'(m,N)$ with $\la \to \nu$ and $|\nu|=|\la|+p$, and the third sum
is empty unless $\la_1=n+k$, and over $\rho \in P(m,N)$ such that
$\rho_1 = n+k$, $\la^* \to \rho$, and $|\rho|=|\la^*|+p$.

Suppose now that $K \geq 4$ is even. Then the quantum cohomology 
ring $\QH(\OG,\Q)$ has a decomposition
\[
\QH(\OG,\Q) = \QH(\OG,\Q)_1 \oplus \QH(\OG,\Q)_{-1}
\]
where $\QH(\OG,\Q)_1$ (respectively $\QH(\OG,\Q)_{-1}$) is the
$\Q[q]$-submodule spanned by the classes $\wh\tau_\la$ (respectively
$\wt\tau_\la$). This is the eigenspace decomposition for the action
of the involution $\iota:\OG\to\OG$ defined in the introduction.
Set $\ov{\OG} = \OG(m-1,N-1)$.
Proposition~\ref{prop-1} has the following generalization.

\begin{prop}\label{qprop-1}
  The $\Q$-linear map $\QH(\ov\OG,\Q) \to \QH(\OG,\Q)_{-1}$ defined by
  $$\sigma_\la\, q^d \mapsto  \wt\tau_{\la+k}\,(-q)^d$$ is an
  isomorphism of $\Q[c_1,\dots,c_{n+k}]$-modules.
\end{prop}
\begin{proof}
We argue as in the proof of Proposition \ref{prop-1} given 
in \S \ref{mainpfs}, using (\ref{eqn:qchernpieri}) and the quantum 
Pieri rule for the products $c_p \cdot \ta_\la$ obtained from 
\cite[Thm.\ 3.4]{BKT1}.
\end{proof}

\noindent
Observe that (\ref{eqn:qchernpieri}) and Proposition \ref{qprop-1} 
determine the products $c_p\cdot \wt{\ta}_\la$ in $\QH(\OG)$ 
for all $p\geq 1$ and even $K\geq 4$.

Assume next that $K=2$, so that $k=1$ and $m=n$. Set $\wh q = q_1+q_2$
and $\wt q = q_1 - q_2$. The following results are direct consequences
of \cite[Thm.\ A.1]{BKT1}. For $\la \in P(n,2n+2)$ and $p \in
[1,n+1]$ such that $p \neq 1$ or $\la \neq (1^n)$, we have
\begin{equation}
\label{gen1}
  c_p \cdot \wh\tau_\la = 
  \sum_{\la \to \mu} 2^{\wh N(\la,\mu)}\, \wh\tau_\mu +
  \sum_{\la \to \nu} 2^{\wh N(\la,\nu)-1}\,
  (\wh\tau_{\wt\nu}\, \wh q - \wt\tau_{\wt\nu}\, \wt q\,) 
  + \sum_{\la^* \to \rho} 2^{\wh N(\la^*,\rho)}\, \wh\tau_{\rho^*}\, q_1 q_2\,,
\end{equation}
where the first sum is classical, the second sum is over $\nu \in
P'(n,2n+2)$ with $\la\to\nu$ and $|\nu|=|\la|+p$, and the third sum is
empty unless $\la_1=n+1$ and over $\rho\in P(n,2n+2)$ such that
$\rho_1=n+1$, $\la^* \to \rho$, and $|\rho|=|\la^*|+p$.

If $\la \in P(n-1,2n+1)$ and $p \in [1,n+1]$ with $p \neq 1$ or $\la
\neq (1^{n-1})$, we have
\begin{equation}
\label{gen2}
  \begin{split}
  c_p \cdot \wt\tau_{(\la,1)} = &\sum_{\la \to \mu} 2^{N(\la,\mu)}\,
  \wt\tau_{(\mu,1)} + \sum_{\la \to \nu}
  2^{N(\la,\nu)-1}\,
  (\wh\tau_{(\wt\nu,1)}\, \wt q - \wt\tau_{(\wt\nu,1)}\, \wh q\,)\\
  &+ \sum_{\la^* \to \rho} 2^{N(\la^*,\rho)}\,
  \wt\tau_{(\rho^*,1)}\, q_1 q_2\,,
  \end{split}
\end{equation}
where the first sum is classical, the second sum is over $\nu \in
P'(n-1,2n+1)$ with $\la \to \nu$ (for $K=3$) and $|\nu|=|\la|+p$, and
the third sum is empty unless $\la_1=n+1$ and over $\rho \in
P(n-1,2n+1)$ such that $\rho_1=n+1$, $\la^* \to \rho$ (for $K=3$), and
$|\rho|=|\la^*|+p$. Moreover, the integers $N(\la,\mu)$,
$N(\la,\nu)$, and $N(\la^*,\rho)$ are computed for $K=3$.

Finally, we have
\begin{equation}
\label{special1}
c_1 \cdot \wh\tau_{(1^n)} = \begin{cases} 2\, \wh\tau_{n+1} + 2\,
  \wh\tau_{(2,1^{n-1})} + \wh q & \text{if $n>1$}\\ 2\, \wh\tau_{n+1}
  + \wh q & \text{if $n=1$}\\
\end{cases}
\end{equation}
and
\begin{equation}
\label{special2} 
c_1 \cdot \wt\tau_{(1^n)} = 
2\, \wt\tau_{(2,1^{n-1})} + \wt q \,.
\end{equation}

\begin{remark}
Proposition~\ref{qprop-1} is valid for $K=2$ and $\OG = \OG(n,2n+2)$
if the quantum cohomology ring of the latter space is replaced with
$\QH(\OG,\Q)/(q_1-q_2)$, with $q := q_1 = q_2$.  Similarly, equation
(\ref{eqn:qchernpieri}) is valid in this quotient ring, except for the
product $c_1 \cdot \wh\tau_{(1^n)}$.  This product is special because
it is the only one that can produce a partition $\nu \in P'(n,2n+2)$
such that $\nu$ has positive type but $\wt\nu$ does not have positive
type.
\end{remark}

\subsection{The stable cohomology ring of $\OG$}

The stable cohomology ring of the orthogonal Grassmannian $\OG(m,N)$
depends only on $K=N-2m$, and is denoted $\IH(\OG_K)$. The ring
$\IH(\OG_K)$ is defined as the inverse limit in the category of graded
rings of the system
\[
\cdots \leftarrow \HH^*(\OG(m,N),\Z) \leftarrow
\HH^*(\OG(m+1,N+2),\Z) \leftarrow \cdots
\]
We set $\IH(\OG_k(\text{odd})) := \IH(\OG_{2k+1})$ and 
$\IH(\OG_k(\text{even})) := \IH(\OG_{2k})$. The ring 
$\IH(\OG_k(\text{odd}))$ was studied in \cite{BKT2, BKT3}.

It follows from \cite[Thm.\ 3.2]{BKT1} that $\IH(\OG_k(\text{even}))$
may be presented as a quotient of the polynomial ring
$\Z[\ta_1,\ldots,\ta_{k-1},\ta_k,\ta_k',\ta_{k+1},\ldots]$ modulo the
relations
\begin{gather}
\label{stabrel1}
\tau_r^2+ \sum_{i=1}^r(-1)^i\tau_{r+i}c_{r-i} = 0 \ \ \ \text{for
$r>k$}, \\
\label{stabrel2}
\tau_k\tau'_k+\sum_{i=1}^k (-1)^i\tau_{k+i}\tau_{k-i} = 0,
\end{gather}
where the $c_i$ obey the equations (\ref{ctotau}).  

The stable cohomology ring $\IH(\OG_k(\mathrm{even}))$ has a free
$\Z$-basis of Schubert classes $\tau_\la$, one for each typed
$k$-strict partition $\la$.  There is a natural surjective ring
homomorphism
\[
\IH(\OG_k(\mathrm{even}))\to\HH^*(\OG(n+1-k,2n+2),\Z)
\]
that maps $\tau_\la$ to $\tau_\la$, when $\la\in\wt{\cP}(k,n)$, and to
zero, otherwise.  All the conclusions of \S \ref{classgiam} remain
true for the ring $\IH(\OG_k(\text{even}))$, with no restrictions on
the size of the (typed) $k$-strict partitions involved.

\subsection{Proof of Theorem \ref{qgiam}}

We first generalize two results from \cite{BKT3}.

\begin{lemma}\label{lem:only}
  Let $\la$ be a $k$-strict partition contained in an $m \times (n+k)$
  rectangle.  Then the stable Giambelli polynomial $R^\la\, c_\la$ for
  $[Y_\la]$ in $\IH(\OG_K)$ involves only Chern classes $c_p$ with $p
  \leq 2n+2k-1$.
\end{lemma}
\begin{proof}
  The proof is the same as that of \cite[Cor.\ 1]{BKT3}.
\end{proof}

For any abelian group $A$, let $A_\Q=A\otimes_\Z\Q$.

\begin{prop}\label{prop:recurse}
  Let $\la$ be a $k$-strict partition.  Then there exist unique
  coefficients $a_{p,\mu} \in \Q$ for $p \geq \la_1$ and $(p,\mu)$ a
  $k$-strict partition, such that the recursive identity
  \[
  [Y_\la] = \sum_{p \geq \la_1}
  \sum_{\mu\,:\,(p,\mu)\,k\text{-}\mathrm{strict}}
  a_{p,\mu}\, c_p\, [Y_\mu]
  \]
  holds in the stable cohomology ring $\IH(\OG_K)_{\Q}$.  Furthermore,
  $a_{p,\mu}=0$ whenever $\mu \not\subset \la^*$, or when $\la$ is
  contained in an $m \times (n+k)$ rectangle and $p \geq 2n+2k$.
\end{prop}
\begin{proof}
  The proof is the same as that of \cite[Prop.\ 3]{BKT3}.
\end{proof}

Next, we give the even orthogonal analogue of \cite[Prop.\ 5]{BKT3}. If 
$\la$ is a $k$-strict partition of positive type, we let $\la-k$ denote 
the partition obtained by removing one part equal to $k$ from $\la$.

\begin{prop}\label{piprop}
  There exists a unique ring homomorphism
  \[
  \pi : \IH(\OG_k(\mathrm{even})) \to \QH(\OG(n+1-k,2n+2))
  \]
  such that the following relations are satisfied:
  \begin{align*}
  \pi(\tau_i) &=
  \begin{cases}
    \tau_i & \text{if $1\leq i \leq n+k$}, \\
    0 & \text{if $n+k< i < 2n+2k$}, \\
    0 & \text{if $i$ is odd and $i>2n+2k$},
  \end{cases}\\
  \pi(\tau'_k) &= \tau'_k.
  \end{align*}
  Furthermore, we have $\pi(\tau_\la) = \tau_\la$ for each $\la \in
  \wt{\cP}(k,n)$.
\end{prop}
\begin{proof}
  The relations \eqref{stabrel1}--\eqref{stabrel2}
  for $r \geq n+k$ uniquely specify
  the values $\pi(\tau_i)$ for even integers $i \geq 2n+2k$.  We
  must show that the remaining relations for $k<r<n+k$ are mapped to
  zero by $\pi$. When $k<n-1$ the individual terms in
  these relations carry no $q$ correction.
  It remains only to consider the
  case $k=n-1$, and $\OG=\OG(1,2n+2)$ is a quadric.
  The relation in degree $2n$ is treated in \cite[Thms.\ 3.5, A.2]{BKT1},
  and the other relations are handled similarly, with the expression in
  degree $2r$ ($n<r<2n-1$) yielding a coefficient of $qc_{2(r-n)+1}$ of
  $1-2+2-\cdots \pm2\mp1=0$.

  Let $\la$ be any $k$-strict partition contained in an
  $(n+1-k)\times(n+k)$ rectangle. It follows from
  Proposition~\ref{prop:recurse} that
  \begin{equation}\label{eqn:rechat}
    \wh\tau_\la = \sum_{\la_1 \leq p < 2n+2k,\ \mu \subset \la^*} a_{p,\mu}\,
    c_p\, \wh\tau_\mu 
  \end{equation}
  holds in $\IH(\OG_k(\mathrm{even}))_{\Q}$. If $\la$ has positive type,
  then Proposition~\ref{prop:recurse} applied to $\OG(n-k,2n+1)$ also
  gives
  \begin{equation}
  \label{eqn:sigma}
  \sigma_{\la-k} = \sum_{p, \mu} a'_{p,\mu}\, c_p\, \sigma_\mu
  \end{equation}
  in $\IH(\OG_k(\mathrm{odd}))_{\Q}$. Using Proposition~\ref{prop-1}, equation
  (\ref{eqn:sigma}) implies that
  \begin{equation}\label{eqn:recdiff}
    \wt\tau_\la = \sum_{p,\mu} a'_{p,\mu}\, c_p\, \wt\tau_{\mu+k}
  \end{equation}
  holds in $\IH(\OG_k(\mathrm{even}))_{\Q}$.  Observe that we have $\mu
  \subset (\la-k)^*$ for each partition $\mu$ appearing in the sum
  (\ref{eqn:recdiff}), and hence also $\mu+k \subset \la^*$.  

  To prove that $\pi(\tau_\la) = \tau_\la$, it is enough to show that
  $\pi(\wh\tau_\la) = \wh\tau_\la$ and $\pi(\wt\tau_\la) =
  \wt\tau_\la$. We argue by induction on $\ell(\la)$, the case
  $\ell(\la)=1$ being clear. When $\la$ has more than one part, we
  apply the ring homomorphism $\pi$ to both sides of
  (\ref{eqn:rechat}) and (\ref{eqn:recdiff}) and use the inductive 
  hypothesis to show that 
\[
   \pi(\wh\tau_\la) = \sum_{p, \mu} a_{p,\mu}\,
    c_p\, \wh\tau_\mu 
\ \ \ \ \mathrm{and} \ \ \ \ 
   \pi(\wt\tau_\la) = \sum_{p,\mu} a'_{p,\mu}\, c_p\, \wt\tau_{\mu+k}
\]
hold in $\QH(\OG(n+1-k,2n+2),\Q)$. The quantum Pieri rules of \S
\ref{S:qpieri} imply that none of the products appearing in these sums
involve $q$ correction terms. Specifically, when $K \geq 3$ this
follows from (\ref{eqn:qchernpieri}) and Proposition \ref{qprop-1},
while the case $K=2$ uses (\ref{gen1}), (\ref{gen2}),
(\ref{special1}), and (\ref{special2}). This completes the proof.
\end{proof}

Theorem \ref{qgiam} follows immediately from
Proposition~\ref{piprop} and Lemma~\ref{lem:only}.

\section{Eta Polynomials}
\label{etasec}

\subsection{}
\label{theta1}
Given any power series $\sum_{i \geq 0} c_i t^i$ in the variable $t$
and an integer sequence $\al = (\al_1,\al_2,\dots,\al_\ell)$, we write
$c_\al = c_{\al_1} c_{\al_2} \cdots c_{\al_\ell}$ and set $R\,c_{\al}
= c_{R\al}$ for any raising operator $R$.  We will always work with
power series with constant term 1, so that $c_0=1$ and $c_i=0$ for
$i<0$.

Let $x=(x_1,x_2,\ldots)$ be a list of commuting
independent variables and let $\Lambda=\Lambda(x)$ be the ring of
symmetric functions in $x$. Consider the generating series
\[
\prod_{i=1}^{\infty}(1+x_it) = \sum_{r=0}^{\infty}
e_r(x)t^r
\]
for the elementary symmetric functions $e_r$.  If $\la$ is any partition,
let $\la'$ denote the partition conjugate to $\la$, and define the 
Schur $S$-function $s_{\la'}(x)$ by the equation
\[
s_{\la'} = \prod_{i<j} (1-R_{ij})\, e_\la .
\]
Moreover, define the functions $q_r(x)$ by the generating series
\[
\prod_{i=1}^{\infty}\frac{1+x_it}{1-x_it} 
= \sum_{r=0}^{\infty}q_r(x)t^r
\]
and let $\Gamma = \Z[q_1,q_2,\ldots]$.
Given any strict partition $\la$ of length $\ell(\la)$, the Schur $Q$-function 
$Q_\la(x)$ is defined by the equation
\[
Q_\la = \prod_{i<j}\frac{1-R_{ij}}{1+R_{ij}}\, q_\la
\]
and the $P$-function is given by $P_\la(x)= 2^{-\ell(\la)}\,Q_\la(x)$.

Fix an integer $k\geq 1$, let $y=(y_1,\ldots,y_k)$ and
$\Lambda_y=\Z[y_1,\ldots,y_k]^{S_k}$. For each integer $r$, define
$\ti_r=\ti_r(x\,;y)$ by
\[
\ti_r = \sum_{i\geq 0} q_{r-i}(x) e_i(y).
\]
We let $\Gamma_y=\Gamma\otimes_\Z\Lambda_y$ and
$\Gamma^{(k)}$ be the subring of $\Gamma_y$ generated by the $\ti_r$:
\[
\Gamma^{(k)}= \Z[\ti_1,\ti_2,\ti_3,\ldots].
\]
According to \cite[(19)]{BKT2}, we have 
\begin{equation}
\label{t-to-e}
\ti_r^2 + 2\sum_{i=1}^r(-1)^i\ti_{r+i}\ti_{r-i} = e_r(y^2)
\end{equation}
for any integer $r$, where $y^2$ denotes $(y_1^2,\ldots,y_k^2)$.

\begin{prop}
\label{Gyprop}
The $\ti_\la$ for $\la$ a strict partition form a free $\Lambda_y$-basis 
of $\Gamma_y$.
\end{prop}
\begin{proof}
It is known e.g.\ from \cite[III.(8.6)]{M} that the $q_\la(x)$ for $\la$
strict form a $\Z$-basis of $\Gamma$, and therefore also a $\Lambda_y$-basis
of $\Gamma_y$. Since 
\[
\ti_{\la}(x\,;y)= q_\la(x) + \sum_{\alpha\neq 0} q_{\la-\alpha}(x)e_{\alpha}(y)
\]
with the sum over nonzero compositions $\al$, we deduce that the 
$\ti_\la$ for $\la$ strict also form a $\Lambda_y$-basis of $\Gamma_y$.
\end{proof}

Following \cite[Prop.\ 5.2]{BKT2}, the $\ti_{\la}$ for $\la$
$k$-strict form a $\Z$-basis of $\Gamma^{(k)}$.  For any $k$-strict
partition $\la$, the theta polynomial $\Ti_\la(x\,;y)$ is defined by
$\Ti_{\la} =\wt{R}^\la\ti_\la$.  The polynomials $\Theta_\la$ for all
$k$-strict partitions $\la$ form another $\Z$-basis of $\Gamma^{(k)}$.

\subsection{}
\label{eta1}
Recall that $P_0=1$ and for each integer $r\geq 1$, we have $P_r = q_r/2$.
Set
\[
\eta_r(x\,;y) = \begin{cases}
e_r(y) + 2\sum_{i=0}^{r-1} P_{r-i}(x)e_i(y) & \text{if $r<k$}, \\
 \sum_{i=0}^r P_{r-i}(x) e_i(y) & \text{if $r\geq k$}
\end{cases}
\]
and $\eta'_k(x\,;y) =  \sum_{i=0}^{k-1} P_{k-i}(x) e_i(y)$. 
Observe that we have, for any $r\geq 0$, 
\begin{equation}
\label{thetatoeta}
\ti_r=
\begin{cases}
\eta_r &\text{if $r< k$},\\
\eta_k+\eta_k' &\text{if $r=k$},\\
2\eta_r &\text{if $r> k$},
\end{cases}
\end{equation}
while $\eta_k-\eta'_k=e_k(y)$. Define the ring of eta polynomials
\[
B^{(k)} = \Z[\eta_1,\ldots, \eta_{k-1},\eta_k,\eta'_k,\eta_{k+1}\ldots].
\]

\begin{prop}
\label{basispr}
The $\Q$-algebra $B^{(k)}_\Q$ is a free $\Gamma^{(k)}_\Q$-module with basis
$1$, $e_k(y)$.
\end{prop}
\begin{proof}
The definition implies that 
$B^{(k)}_\Q=\Gamma^{(k)}_\Q[e_k(y)]$. We claim that 
\[
\Gamma^{(k)}_\Q\cap e_k(y)\Gamma^{(k)}_\Q=0.
\] 
Indeed, we know that the $\ti_\la$ for $k$-strict partitions $\la$
form a $\Q$-basis of $\Gamma^{(k)}_\Q$.  We deduce from (\ref{t-to-e})
and Proposition \ref{Gyprop} that for each $k$-strict partition $\la$,
there is a unique expression
\begin{equation}
\label{k-to-strict}
\ti_\la(x\,;y) = \sum_{\mu,\nu}a_{\mu\nu}\,\ti_\mu(x\,;y)e_\nu(y^2)
\end{equation}
where the $a_{\mu\nu}$ are integers and the sum is over strict
partitions $\mu$ and partitions $\nu$ with $\nu_1\leq k$ such that
$|\mu|+2|\nu|=|\la|$. The claim follows from this and the fact that
the $\Z$-linear submodules of $\Lambda_y$ spanned by the $e_\nu(y^2)$ and
$e_k(y)e_\nu(y^2)$, respectively, have trivial intersection.  We
deduce that $B^{(k)}_\Q$ is the $\Gamma^{(k)}_\Q$-algebra generated by
$e_k(y)$ modulo the quadratic relation (\ref{t-to-e}). The proposition
now follows by elementary algebra.
\end{proof}

Let $\la$ be a typed $k$-strict partition and let $R$ be any finite
monomial in the operators $R_{ij}$ which appears in the expansion of
the power series $R^\la$ in (\ref{Req}).  If $\type(\la)=0$, then set
$R \star \ti_{\la} = \ti_{R \,\la}$. Suppose that $\type(\la)>0$, let
$d = \ell_k(\la)+1$ be the index such that $\la_d=k < \la_{d-1}$, and
set $\wh{\al} = (\al_1,\ldots,\al_{d-1},\al_{d+1},\ldots,\al_\ell)$
for any integer sequence $\al$ of length $\ell$. If $R$ involves any
factors $R_{ij}$ with $i=d$ or $j=d$, then let $R \star \ti_{\la} =
\frac{1}{2}\,\ti_{R \,\la}$. If $R$ has no such factors, then let
\[
R \star \ti_{\la} = \begin{cases}
\eta_k \,\ti_{\wh{R \,\la}} & \text{if  $\,\type(\la) = 1$}, \\
\eta'_k \, \ti_{\wh{R \,\la}} & \text{if  $\,\type(\la) = 2$}.
\end{cases}
\]

\begin{defn}
\label{etadef}
For any typed $k$-strict partition $\la$, the {\em eta polynomial}
$\Eta_\la=\Eta_\la(x\,;y)$ is the element of $B^{(k)}$ defined by the
raising operator formula 
\[
\Eta_\la =2^{-\ell_k(\la)}R^{\la} \star \ti_{\la}.  
\]
The {\em type} of the polynomial $\Eta_\la$ is the same
as the type of $\la$.
\end{defn}

If $\la$ is a $k$-strict partition, we define the polynomials
$\Eta_\la$, $\Eta'_\la$, $\wh{\Eta}_\la$, and $\wt{\Eta}_\la$ using
the same conventions as in \S \ref{mainpfs} in the case of Schubert
classes. Note that
\[
\wh{\Eta}_\la = 2^{-\ell_k(\la)} R^\la \,\ti_\la
=\begin{cases}
\Eta_\la + \Eta'_\la & \text{if $\la$ has positive type}, \\
\ \quad \Eta_\la & \text{otherwise},
\end{cases}
\]
while if $\la$ has positive type then 
\[
\wt{\Eta}_\la = \Eta_\la - \Eta'_\la =
 2^{-\ell_k(\la)} e_k(y) \,\wt{R}^{\la-k} \,\ti_{\la-k} =
 2^{-\ell_k(\la)} e_k(y) \, \Ti_{\la-k}\,.
\]
The raising operator expression for $\wh{\Eta}_\la$ implies that
$\Gamma^{(k)}_\Q = \bigoplus_\la \Q\,\wh{\Eta}_\la$, summed over all
$k$-strict partitions $\la$. It follows from Proposition \ref{basispr}
that there is a direct sum decomposition
\[
B^{(k)}_\Q =  \ 
\Gamma^{(k)}_\Q \oplus e_k(y)\,\Gamma^{(k)}_\Q
=\bigoplus_{\la \ \text{$k$-strict}} \Q\,\wh{\Eta}_\la \oplus
\bigoplus_{\substack{\la \ \text{$k$-strict} \\ \text{of positive
      type}}} \Q\,\wt{\Eta}_\la .
\]

For any typed $k$-strict partition $\la$, we have 
\begin{equation}
\label{whwteta}
\Eta_\la = \begin{cases}
\wh{\Eta}_\la & \text{if $\type(\la)=0$}, \\
(\wh{\Eta}_\la + \wt{\Eta}_\la)/2 & \text{if $\type(\la)=1$}, \\
(\wh{\Eta}_\la - \wt{\Eta}_\la)/2 & \text{if $\type(\la)=2$}
\end{cases}
\end{equation}
in analogy with (\ref{whwtta}).
We deduce that $B^{(k)}$ is isomorphic to the stable cohomology ring
$\IH(\OG_k(\text{even}))$ via the map which sends $\Eta_\la$ to
$\ta_\la$, and that the polynomials $\Eta_\la$ indexed by typed
$k$-strict partitions $\la$ form a $\Z$-basis of $B^{(k)}$. Indeed,
the relations (\ref{stabrel1}) and (\ref{stabrel2}) are satisfied by
the $\eta_r$, $r\geq 1$ and $\eta'_k$, as follows immediately from the
equations (\ref{t-to-e}) for the values $r\geq k$. This completes the
proof of Theorem \ref{productthm}.

\subsection{}
\label{theta4}
Define the elements $S_{\la}(x\,;y)$ and $Q_\la(x\,;y)$ of 
$\Gamma^{(k)}$ by the equations
\[
S_{\la}(x\,;y) = 
\prod_{i<j}(1-R_{ij})\, \ti_\la = \det(\ti_{\la_i+j-i})_{i,j}
\]
and 
\[
Q_{\la}(x\,;y) = \prod_{i<j}\frac{1-R_{ij}}{1+R_{ij}}\, \ti_\la.
\]
These polynomials were studied in \cite[Thm.\ 3 and Prop.\ 5.9]{BKT2}.
Using equation (\ref{whwteta}), we easily obtain the following result.

\begin{prop}  
\label{propSQ}
Let $\la$ be a typed $k$-strict partition of length $\ell$.

\medskip
\noindent
{\em(a)} If $\la_i + \la_j < 2k+j-i$ for all $i<j$, then we have
\[
H_\la(x\,;y) = 
\begin{cases}
S_\la(x\,;y) & \text{{\em if} $\la_1<k$}, \\
\frac{1}{2}\,S_\la(x\,;y) + \frac{1}{2}\, e_k(y) \,S_{\la-k}(x\,;y) & 
 \text{{\em if} $\type(\la)=1$}, \\
\frac{1}{2}\,S_\la(x\,;y) - \frac{1}{2}\,e_k(y)\,S_{\la-k}(x\,;y) & 
 \text{{\em if} $\type(\la)=2$}, \\
\frac{1}{2}\,S_\la(x\,;y) & \text{{\em if} $\la_1>k$}.
\end{cases}
\]

\noindent
{\em(b)} If $\la_i + \la_j \geq 2k+j-i$ for all $i<j\leq \ell$, 
then we have
\[
H_\la(x\,;y) = 
\begin{cases}
2^{-\ell}\,Q_\la(x\,;y) & \text{{\em if} $\la_\ell > k$}, \\
2^{-\ell}\,Q_\la(x\,;y) + 2^{-\ell} \, e_k(y)\,Q_{\la-k}(x\,;y) & 
 \text{{\em if} $\type(\la)=1$}, \\
2^{-\ell}\,Q_\la(x\,;y) - 2^{-\ell} \, e_k(y)\,Q_{\la-k}(x\,;y) & 
 \text{{\em if} $\type(\la)=2$}, \\
2^{1-\ell}\,Q_\la(x\,;y) & \text{{\em if} $\la_\ell < k$}.
\end{cases}
\]
\end{prop}

\section{Schubert polynomials for even orthogonal Grassmannians}\label{BH}

\subsection{}\label{bh1}

In this section, we prove that the polynomials $\Eta_\la(x\,;y)$ are
special cases of Billey-Haiman type D Schubert polynomials
$\DS_w(x,z)$. Let $\wt{W}_{n+1}$ be the Weyl group for the root system
of type $\text{D}_{n+1}$. The elements of $\wt{W}_{n+1}$ may be
represented as signed permutations of the set $\{1,\ldots,n+1\}$; we
will denote a sign change by a bar over the corresponding entry. The
group $\wt{W}_{n+1}$ is generated by the simple transpositions
$s_i=(i,i+1)$ for $1\leq i\leq n$, and an element $s_0$ which acts on
the right by
\[
(u_1,u_2,\ldots,u_{n+1})s_0 = (\ov{u}_2, \ov{u}_1, u_3 , \ldots, u_{n+1}).
\]
Set $\wt{W}_\infty = \bigcup_n\wt{W}_n$ and let $w \in \wt W_\infty$.  
A {\em reduced factorization} of $w$ is a
product $w = u v$ in $\wt W_\infty$ such that $\ell(w) = \ell(u) +
\ell(v)$.  A {\em reduced word} of $w\in\wt{W}_\infty$ is a sequence
$a_1\cdots a_\ell$ of elements in $\N$ such that $w=s_{a_1}\cdots
s_{a_\ell}$ and $\ell=\ell(w)$.  
If we convert all the $0$'s which appear in the reduced word
$a_1\ldots a_r$ to $1$'s, we obtain a {\em flattened word} of $w$.
For example, $23012$ is a reduced word of $\ov{1}4\ov{3}2$, and
$23112$ is the corresponding flattened word (note that the flattened 
word need not be reduced). We say that $w$ has a {\em descent} at
position $r\geq 0$ if $\ell(ws_r)<\ell(w)$, where $s_r$ is the simple
reflection indexed by $r$.  

For $k \neq 1$, an element $w\in\wt{W}_\infty$ is $k$-Grassmannian if
$\ell(ws_i)=\ell(w)+1$ for all $i\neq k$. We say that $w$ is
$1$-Grassmannian if $\ell(ws_i)=\ell(w)+1$ for all $i\geq 2$.  The
elements of $\wt{W}_{n+1}$ index the Schubert classes in the
cohomology ring of the flag variety $\SO_{2n+2}/B$, which contains
$\HH^*(\OG(n+1-k,2n+2),\Z)$ as the subring spanned by Schubert classes
given by $k$-Grassmannian elements.  In particular, each typed
$k$-strict partition $\la$ in $\wt{\cP}(k,n)$ corresponds to a
$k$-Grassmannian element $w_\la \in \wt W_{n+1}$ which we proceed to
describe; more details and relations to other indexing conventions can
be found in \cite[\S 6]{T1}.

Given any typed $k$-strict partition $\la$, we let $\la^1$ be the strict
partition obtained by removing the first $k$ columns of $\la$, and let
$\la^2$ be the partition of boxes contained in the first $k$ columns
of $\la$.
\[ \pic{0.50}{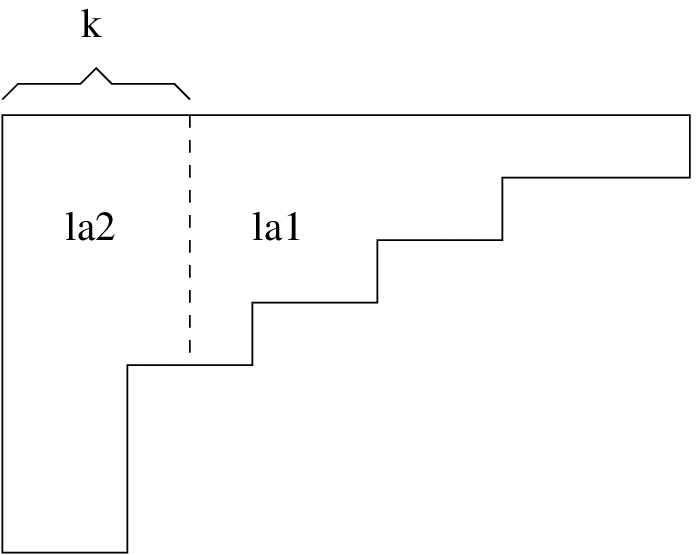} \]
A typed $k$-strict partition $\la$ belongs to $\wt{\cP}(k,n)$ if and
only if its Young diagram fits inside the shape $\Pi$ obtained by
attaching an $(n+1-k)\times k$ rectangle to the left side of a staircase
partition with $n$ rows.  When $n=7$ and $k=3$, this shape looks as
follows.
\[ \Pi \ \ = \ \ \ \raisebox{-36pt}{\pic{.6}{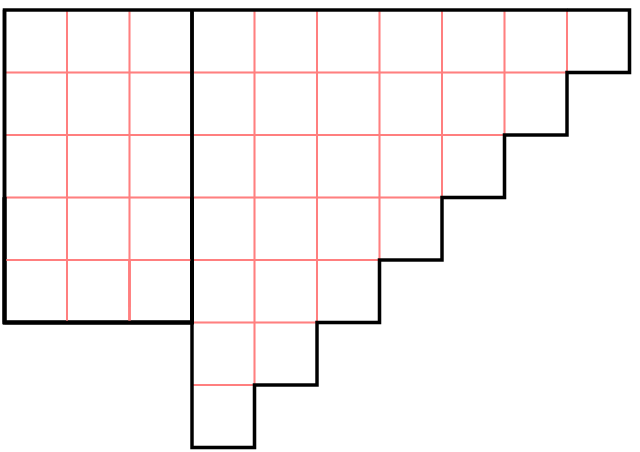}} \] The boxes of the
staircase partition that are outside $\la$ form south-west to
north-east diagonals.  Such a diagonal is called {\em related} if it
is $K$-related to one of the bottom boxes in the first $k$ columns of
$\la$, or to any box $[0,i]$ for which $\la_1 < i \leq k$; the
remaining diagonals are {\em non-related}.  Let $r_1 < r_2 < \dots <
r_k$ denote the lengths of the related diagonals, let $u_1 < u_2 <
\dots < u_t$ be the lengths of the non-related diagonals, and set
$p=\ell(\la^1) = \ell_k(\la)$.  If $\type(\la)$ is non-zero, then $t =
n-k-p$.  If $\type(\la)=1$, then the $k$-Grassmannian element
corresponding to $\la$ is given by
\[
w_\la =
(r_1+1,\dots,r_k+1, \ov{(\la^1)_1+1}, \dots, \ov{(\la^1)_p+1},\wh{1},
u_1+1,\dots,u_{n-k-p}+1) \,,
\] 
while if $\type(\la)=2$, then 
\[
w_\la =
(\ov{r_1+1},\dots,r_k+1, \ov{(\la^1)_1+1}, \dots, \ov{(\la^1)_p+1},\wh{1},
u_1+1,\dots,u_{n-k-p}+1) \,.
\]
Here we use the convention that $\wh{1}$ is equal to either $1$ or
$\ov{1}$, determined so that $w_\la$ contains an even number of barred
integers.  Finally, if $\type(\la)=0$, then $r_1=0$, i.e., one of the
related diagonals has length zero.  In this case we have $t = n-k-p+1$
and
\[
w_\la =
(\wh{1},r_2+1, \dots,r_k+1, \ov{(\la^1)_1+1}, \dots, \ov{(\la^1)_p+1},
u_1+1,\dots,u_{n+1-k-p}+1) \,.
\]
The element $w_\la \in \wt{W}_\infty$ depends on $\la$ and $k$, but is
independent of $n$.

\begin{example}
  The element $\la = (7,4,3,2) \in \wt{\cP}(3,7)$ of type $2$
  corresponds to $w_\la = (\ov{3},6,7,\ov{5},\ov{2},\ov{1},4,8)$.
  \[
  \lambda \ = \ \ \raisebox{-53pt}{\pic{.6}{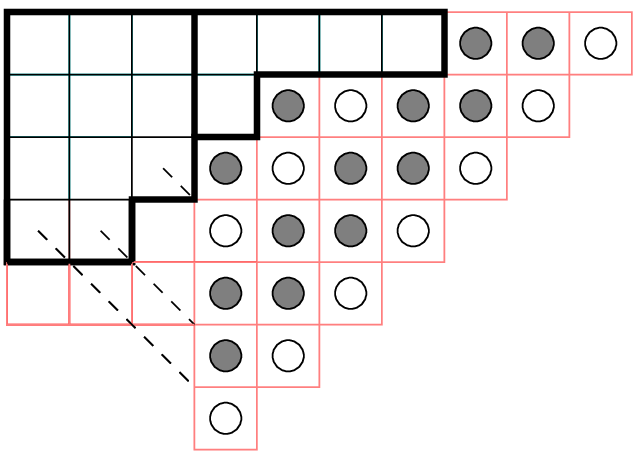}}
  \]
\end{example}

\subsection{}\label{bh2}

We say that a sequence $a=(a_1,\ldots,a_m)$ is {\em unimodal} if for some
$r \leq m$, we have
\[
a_1 > a_2 > \cdots > a_{r-1} \geq a_r < a_{r+1} < \cdots < a_m,
\]
and if $a_{r-1}=a_r$ then $a_r=1$.

Let $w\in \wt{W}_n$ and $\la$ be a Young diagram with $r$ rows and 
$\ell(w)$ boxes. A {\em Kra\'skiewicz-Lam tableau} for $w$ of shape $\la$
is a filling $T$ of the boxes of $\lambda$ with positive integers
in such a way that

\medskip
\noindent
a) If $t_i$ is the sequence of entries in the $i$-th row of $T$,
reading from left to right, then the row word $t_r\ldots t_1$ is
a flattened word for $w$.

\medskip
\noindent
b) For each $i$, $t_i$ is a unimodal subsequence of maximum length
in $t_r \ldots t_{i+1} t_i$.

\medskip
\noin
If $T$ is a Kra\'skiewicz-Lam tableau of shape $\la$ with row
word $a_1\ldots a_\ell$, set $m(T) = \ell(\la)+1-p$, where $p$ is
the number of distinct values of $s_{a_1}\cdots s_{a_j}(1)$ for $0\leq
j\leq \ell$. It follows from \cite[Thm.\ 4.35]{La} that $m(T)\geq 0$.

For each $w\in \wt{W}_{\infty}$ one has a {\em type D Stanley symmetric
function} $E_w(x)$, which is a positive linear combination of Schur
$P$-functions \cite{BH,FK,La}. In particular, Lam \cite{La} has shown that
\begin{equation}
\label{Eeq}
E_w(x) = \sum_{\la} d_w^{\la} \, P_{\la}(x)
\end{equation}
where $d_w^{\la}= \sum_T 2^{m(T)}$, summed over all Kra\'skiewicz-Lam
tableaux $T$ for $w$ of shape $\la$.

\begin{example}\label{maxGr}
  Let $\la$ be a strict partition of length $\ell$ with $\la_1\leq n$,
  and let $\mu$ be the strict partition whose parts are the numbers
  from $1$ to $n+1$ which do not belong to the set
  $\{1,\la_\ell+1,\ldots,\la_1+1\}$.  Then the signed permutation
  \[
  w_{\la}=(\ov{\la_1+1},\ldots,\ov{\la_{\ell}+1},
  \wh{1},\mu_{n-\ell},\ldots,\mu_1)
  \]
  is the $0$-Grassmannian element of $\wt{W}_{n+1}$ corresponding to
  $\la$.  There exists a unique Kra\'skiewicz-Lam tableau $T_\la$ for
  $w_\la$.  This tableau has shape $\la$ and its $i$th row contains
  the integers between $1$ and $\la_i$ in decreasing order; moreover,
  we have $m(T_\la)=0$.  For example,
  \[
  T_{(6,5,2)} =
  \young(654321,54321,21)
  \ .
  \]
  We deduce that $E_{w_{\la}}(x) = P_{\la}(x)$.
\end{example}

\subsection{}\label{bh3}

Following Billey and Haiman \cite{BH}, each $w\in \wt{W}_{\infty}$
defines a type D Schubert polynomial $\DS_w(x,z)$.  Here
$z=(z_1,z_2,\ldots)$ is another infinite set of variables and each
$\DS_w$ is a polynomial in the ring
$A=\Z[P_1(x),P_2(x),\ldots;z_1,z_2,\ldots]$.  The polynomials $\DS_w$
for $w\in \wt{W}_{\infty}$ form a $\Z$-basis of $A$, and their algebra
agrees with the Schubert calculus on orthogonal flag varieties
$\SO_{2n}/B$, when $n$ is sufficiently large.  According to
\cite[Thm.\ 4]{BH}, for any $w\in \wt{W}_n$ we have
\begin{equation}
\label{BHgeneq}
\DS_w(x,z) = \sum_{uv=w} E_u(x) \AS_v(z) \,,
\end{equation}
summed over all reduced factorizations $w = uv$ in $\wt{W}_n$ for
which $v\in S_n$.  Here $\AS_v(z)$ denotes the type A Schubert
polynomial of Lascoux and Sch\"utzenberger \cite{LS}.

\begin{prop}\label{bhprop}
  The ring $B^{(k)}$ of eta polynomials is a subring of the
  ring of Billey-Haiman Schubert polynomials of type D.  For every
  typed $k$-strict partition $\la$ we have $\Eta_{\la}(x\,;y)=
  \DS_{w_{\la}}(x,y)$.
\end{prop}
\begin{proof}
  We first show that the eta polynomial $\eta_r$, $r\geq 1$
  (respectively $\eta'_k$) agrees with the Billey-Haiman Schubert
  polynomial indexed by the $k$-Grassmannian element $w_{(r)}\in
  \wt{W}_\infty$ corresponding to $\la = (r)$ (respectively, by
  $w'_{(k)}$ corresponding to $\la=(k)$ with $\type(\la)=2$).  One
  sees that $w_{(r)}$ has a reduced word given by $(k-r+1, k-r+2,
  \ldots, k)$ when $1 \leq r \leq k$, by $(r-k, r-k-1, \ldots, 1, 0,
  2, 3, \ldots, k)$ when $r > k$, and that $w'_{(k)}$ has the reduced
  word $(0,2,\ldots,k)$.  It follows that if $w_{(r)} = u v$ is any
  reduced factorization of $w_{(r)}$ with $v \in S_\infty$, then $v =
  w_{(i)}$ for some integer $i$ with $0 \leq i \leq k$.  The type A
  Schubert polynomial for $w_{(i)}$ is given by $\AS_{w_{(i)}}(z) =
  e_i(z_1,\dots,z_k)$, and (\ref{Eeq}) implies that the type D Stanley
  symmetric function for $u = w_{(r)} w_{(i)}^{-1}$ is
  \[
  E_u(x) =
  \begin{cases}
    2P_{r-i}(x) & \text{if $r < k$}, \\
    P_{r-i}(x) & \text{if $r \geq k$}
  \end{cases}.
  \]
  We conclude from (\ref{BHgeneq}) that $\DS_{w_{(r)}}(x,z) =
  \eta_r(x\,;z)$ and $\DS_{w'_{(k)}}(x,z) = \eta'_k(x\,;z)$, as
  required. Since the Schubert polynomials $\DS_w$ multiply like the
  Schubert classes on even orthogonal flag varieties, the proposition
  now follows from Theorem \ref{mainthm}.
\end{proof}

\begin{thm}\label{bhthm}
  For any typed $k$-strict partition $\la$, the polynomial $\Eta_\la$ is a
  linear combination of products of Schur $P$-functions and
  $S$-polynomials:
 \begin{equation}
\label{E:thm4}
  \Eta_\la(x\,;y) = \sum_{\mu,\nu}
  d_{\mu\nu}^\la
  P_\mu(x)s_{\nu'}(y)
\end{equation}
where the sum is over partitions $\mu$ and $\nu$ such that $\mu$ is
strict and $\nu \subset \la^2$. Moreover, the coefficients 
$d_{\mu\nu}^\la$ are nonnegative integers, equal to the number of
Kra\'skiewicz-Lam tableaux for $w_\la w_\nu^{-1}$ of shape $\mu$.
\end{thm}
\begin{proof}
Proposition~\ref{bhprop} and (\ref{BHgeneq}) imply that for every
typed $k$-strict partition $\la$ we have
\begin{equation}\label{BHeq}
  \Eta_\la(x\,;y) = \sum_{uv=w_{\la}} E_u(x) \AS_v(y) \,,
\end{equation}
where the sum over all reduced factorizations $w_\la = uv$ in
$W_\infty$ with $v\in S_\infty$.  The right factor $v$ in any such
factorization must be a Grassmannian permutation with descent at
position $k$.  In fact, it is not hard to check that the right reduced
factors of $w_\la$ that belong to $S_\infty$ are permutations $w_\nu$
given by partitions $\nu \subset \la^2$ (and are exactly these
partitions whenever $\type(\la)\neq 2$). We now use (\ref{Eeq}) and
the fact that the Schubert polynomial $\AS_{w_\nu}(y)$ is equal to the
Schur polynomial $s_{\nu'}(y)$.
\end{proof}

If $\la$ is a $k$-strict partition of type $2$, let 
\[
\la^3 = \la^1 + r_1 = ((\la^1)_1,(\la^1)_2,\ldots,r_1,r_1-1,\ldots,2,1)
\]
be $\la^1$ with a part $r_1$ added, and let $\la^4$ be $\la^2$ with
$r_1$ boxes subtracted from the $k$-th column, so that
\[
(\la^4)' = ((\la^2)'_1,\ldots,(\la^2)'_{k-1},(\la^2)'_k - r_1).
\]

\begin{cor}\label{endcor} 
  Let $\la$ be a typed $k$-strict partition. \smallskip

  \noin{\em(a)}
  The homogeneous summand of $\Eta_{\la}(x\,;y)$ of
  highest $x$-degree is the type D Stanley symmetric function
  $E_{w_\la}(x)$, and satisfies $E_{w_{\la}}(x) =
  2^{-\ell_k(\la)}\,R^{\la}\, q_{\la}(x)$. \smallskip

  \noin{\em(b)} The homogeneous summand of $\Eta_{\la}(x\,;y)$ of
  lowest $x$-degree is $P_{\la^1}(x)\,s_{(\la^2)'}(y)$, if $\type(\la)
  \neq 2$, and $P_{\la^3}(x)\,s_{(\la^4)'}(y)$, if $\type(\la)=2$.
\end{cor}
\begin{proof}
  Part (a) follows by setting $y=0$ in (\ref{BHeq}) and also in the
  raising operator expression $H_{\la}(x\,;y) =
  2^{-\ell_k(\la)}\,R^{\la}\star \ti_{\la}(x\,;y)$.  Part (b) is
  deduced from (\ref{E:thm4}), Example~\ref{maxGr}, and the
  observation that $w_\la w_{\la^2}^{-1}$ (respectively, $w_\la
  w_{\la^4}^{-1}$) is the 0-Grassmannian Weyl group element
  corresponding to the strict partition $\la^1$, if $\type(\la)\neq 2$
  (respectively, to the strict partition $\la^3$, if $\type(\la)=2$).
\end{proof}

\appendix

\section{Schubert varieties in orthogonal Grassmannians}
\label{appendix}

Our goal in this section is to give a geometric description of the
Schubert varieties in the orthogonal Grassmannians $\OG(m,N)$, and to
establish the assertions about the subsets $Y_\la$ claimed in the
introduction and Section \ref{mainpfs}. We also correct errors in
the definition of the type D Schubert varieties and degeneracy loci
which appeared in the earlier papers \cite[\S 3.2, p.\ 1713]{KT},
\cite[\S 6.1, p.\ 333]{T1}, and \cite[\S 3.1 and \S 4.3, p.\ 377 and
  389]{BKT1}. In each of these references, the type D Schubert variety
is claimed to be the locus of isotropic subspaces (or flags) $\Sigma$
which satisfies a system of dimension inequalities, obtained by
intersecting $\Sigma$ with the subspaces in a fixed isotropic flag (or
its alternate). The Schubert variety should be defined as the closure
of the corresponding Schubert cell, which consists of the locus of
$\Sigma$ satisfying a system of dimension equalities. See
e.g.\ \cite[\S 2.1]{T2} for a precise definition of the Schubert
varieties in even orthogonal flag varieties in these terms, and the
discussion in \cite{Fu} and \cite[\S 6.1]{FP} for more on this
phenomenon.

We thank Vijay Ravikumar for pointing out that the description of the
type D Schubert varieties and their Bruhat order in \cite[\S
4.3]{BKT1} is wrong, which led us to the above errors. Fortunately,
the mistaken description of the Schubert varieties in \cite{KT, T1,
BKT1} does not affect the correctness of the results of
loc.\ cit.\ outside of \cite[Prop.\ 4.5]{BKT1}, as this description
was never used. We fix the errors in \cite[\S 4]{BKT1} below.

Let $V \cong \C^N$ be a complex vector space equipped with a
non-degenerate symmetric bilinear form $(-,-)$.  Fix $m < N/2$ and let
$\OG=\OG(m,N)$ be the orthogonal Grassmannian of $m$-dimensional
isotropic subspaces of $V$.  This variety has a transitive action of
the group $\SO(V)$ of linear automorphisms that preserve the form on
$V$.  For any subset $A \subset V$ we let $\langle A \rangle \subset
V$ denote the $\C$-linear span of $A$.  Fix an isotropic flag
$F_\bull$ in $V$ and let $B \subset \SO(V)$ be the Borel subgroup
stabilizing $F_\bull$.  The {\em Schubert varieties} in $\OG$ are the
orbit closures of the action of $B$ on $\OG$.  We also fix a basis
$e_1,\dots,e_N$ of $V$ such that $(e_i,e_j) = \delta_{i+j,N+1}$ and
$F_p = \langle e_1,\dots,e_p \rangle$ for each $p$.

We call a subset $P \subset [1,N]$ of cardinality $m$ an {\em index
set} if for all $i,j \in P$ we have $i+j \neq N+1$.  Equivalently,
the subspace $\langle e_p : p \in P \rangle \subset V$ is a point
in $\OG$.  Let $X^\circ_P = B.\langle e_p : p \in P \rangle \subset \OG$
be the orbit of this point, and let $X_P = \ov{X^\circ_P}$ be the
corresponding Schubert variety.  Any point $\Sigma \in \OG$ defines an
index set
\[
P(\Sigma) = \{ p \in [1,N] \mid \Sigma\cap F_p \supsetneq \Sigma\cap
F_{p-1} \} \,,
\]
since no vector in $F_p\ssm F_{p-1}$ is orthogonal to a vector in
$F_{N+1-p}\ssm F_{N-p}$.  It follows from \cite[Lemma~4.1]{BKT1} that
we have $X^\circ_P = \{ \Sigma \in \OG \mid P(\Sigma) = P \}$.  In
particular, the Schubert varieties in $\OG$ are in 1-1 correspondence
with the index sets.  Given any two index sets $P = \{p_1<\dots<p_m\}$
and $Q = \{q_1<\dots<q_m\}$, we write $Q \leq P$ if $q_j \leq p_j$ for
each $j$.

Recall that the integer $n$ is defined so that $N=2n+1$ if $N$ is odd,
$N=2n+2$ if $N$ is even, and $k$ satisfies $n+k = N-m-1$.
Following \cite[\S4]{BKT1} and the introduction, equation (1)
establishes a bijection between the index sets $Q$ for which $n+1
\notin Q$ and the set of $k$-strict partitions $\la$ whose Young
diagram is contained in an $m \times (n+k)$ rectangle.  Notice that
the condition on $Q$ is always true if $N$ is odd.  If $P=\{p_j\}$ is an
arbitrary index set we let $\ov{P}=\{\ov{p}_j\}$ denote the index set with
$\ov{p}_j=n+2$, if $p_j=n+1$, and $\ov{p}_j= p_j$, otherwise. 
Define a Zariski closed subset $Y_P \subset \OG$ by
\[
Y_P = \{ \Sigma \in \OG \mid
\dim(\Sigma \cap F_{\ov{p}_j}) \geq j  \text{ for }
1 \leq j \leq m \} \,.
\]
If $\la$ is the $k$-strict partition corresponding to $\ov{P}$, then
$Y_P$ agrees with the set $Y_\la$ defined in the introduction.  Observe
that $Y_P=\bigcup_{Q\leq \ov{P}} X^\circ_Q$. In particular, 
$X_P \subset Y_P$ for any index set $P$.

\subsection*{Type B}

We assume in this section that $N = 2n+1$ is odd.

\begin{prop}
\label{closureB}
For any index set $P=\{p_1 < \cdots < p_m\}\subset [1,2n+1]$, we have
\[
   X_P= Y_P=\{ \Sigma \in \OG \mid \dim(\Sigma \cap
   F_{p_j}) \geq j,  \ \ \forall \,  1 \leq j \leq m \} \,. 
\]
For any two index sets $P$ and $Q$, we have $X_Q\subset X_P$ if and
only if $Q\leq P$.
\end{prop}

\begin{proof}
  It suffices to prove that $Q \leq P$ implies $X_Q \subset X_P$.
  Assuming $Q < P$, it is enough to construct an index
  set $P'$ such that $Q \leq P' < P$ and $X_{P'} \subset X_P$.

 Choose $j$ minimal such that $q_j < p_j$, and notice that
 $[q_j,p_j-1]\cap P=\emptyset$.  If some integer $x \in [q_j,p_j-1]$
 satisfies that $x \neq n+1$ and $N+1-x \notin P$, then set $P' =
 \{p_1,\dots,p_{j-1},x, p_{j+1},\dots,p_m\}$, and observe that $Q \leq
 P' < P$.  Define a morphism of varieties $\bP^1 \to \OG$ by
  \begin{equation}\label{st1}
    [s:t] \mapsto \Sigma_{[s:t]} =
    \langle e_{p_1},\dots,e_{p_{j-1}}, s\,e_x + t\,e_{p_j},
    e_{p_{j+1}}, \dots, e_{p_m} \rangle \,.
  \end{equation}
  Since $\Sigma_{[1:0]} \in X^\circ_{P'}$ and $\Sigma_{[s:t]} \in
  X^\circ_P$ for $t \neq 0$, it follows that $X^\circ_{P'} \subset X_P$.

Otherwise, we must have $N+1-x\in P$ for all $x\in
[q_j,p_j-1]\ssm\{n+1\}$.  We deduce that $q_j\leq n$ and $p_j\leq
n+2$. If $p_j=n+2$, then we may use the set
$P'=\{p_1,\dots,p_{j-1},n,p_{j+1},\dots,p_m\}$ and the morphism
\[
[s:t] \mapsto \Sigma_{[s:t]} =
  \langle e_{p_1},\dots,e_{p_{j-1}}, s^2\, e_n + 2st\,e_{n+1}-2t^2\, e_{n+2},
  e_{p_{j+1}}, \dots, e_{p_m} \rangle
\]
from $\bP^1$ to $\OG$ to conclude as above that $X^\circ_{P'} \subset X_P$.

We may therefore assume that $q_j<p_j\leq n$. Set $P' = (P
\smallsetminus \{p_j,N+1-q_j\}) \cup \{q_j,N+1-p_j\}$.  We then use
the morphism $\bP^1 \to \OG$ given by
\begin{equation}
\label{st2}
[s:t] \mapsto \Sigma_{[s:t]} = \langle e_p : p \in P\cap P' \rangle
 \oplus \langle s\, e_{q_j} + t\, e_{p_j}, s\, e_{N+1-p_j} - t\,
  e_{N+1-q_j} \rangle
\end{equation}
to show that $X^\circ_{P'} \subset X_P$.  We finally
check that $Q\leq P'< P$. The relation $P' < P$ is true because
  $q_j < p_j$ and $N+1-p_j < N+1-q_j$.  Set $u = p_j-q_j$ and choose
  $v>j$ such that $p_v = N+1-q_j$.  Then we have $q_v < p_v$ and
  $p_{v-i} = p_v-i$ for $i \in [0,u-1]$.  It follows that $q_{v-i}
  \leq q_v-i \leq p_v-i-1 = p'_{v-i}$ for $i \in [0,u-1]$.  Since $q_j
  = p'_j$ and $p'_r = p_r$ for $r \notin \{j\} \cup [v-u+1,v]$, this
  implies that $Q \leq P'$.
\end{proof}

\subsection*{Type D}

We assume in this section that $N = 2n+2$ is even.  Following
\cite[\S4.3]{BKT1} we agree that every index set $P \subset [1,2n+2]$
has a {\em type}, which is an integer $\type(P) \in \{0,1,2\}$.  If $P
\cap \{n+1,n+2\} = \emptyset$, then the type of $P$ is zero.
Otherwise, $\type(P)$ is equal to $1$ plus the parity of the number of
integers in $[1,n+1] \ssm P$.  As in loc.\ cit.\ and the introduction,
there is a type preserving bijection between the index sets $P \subset
[1,2n+2]$ and the typed $k$-strict partitions $\la$ in $\wt\cP(k,n)$.
If $P$ corresponds to $\la$, then $X_P$ is also denoted $X_\la$.

For any index set $P$, let $[P] = P \cup \{N+1-p \mid p \in P\}$.  A
pair $(Q,P)$ of index sets is called {\em critical\/} if there exists
an integer $c \leq n+1$ such that $[c,n+1] \subset [P] \cap [Q]$ and
$\# Q \cap [1,c-1] = \# P \cap [1,c-1]$.  Such an integer $c$ is then
called a {\em critical index}.  We will write $Q \preceq P$ if (i) $Q
\leq P$ and (ii) if $(Q,P)$ is critical then $\type(Q) = \type(P)$.

Notice that if $q_j=n+1$ and $p_j=n+2$ for some $j$, then $n+1$ is a
critical index for $(Q,P)$ and $\type(Q) \neq \type(P)$, so $Q
\not\preceq P$.  Let $\iota$ be the involution on index sets that
interchanges $n+1$ and $n+2$.  We have $Q \preceq P$ if and only if
$\iota(Q) \preceq \iota(P)$.  We note that $Q \le \overline{P}$ if and
only if $Q \preceq P$ or $Q \preceq \iota(P)$.

\begin{prop}
\label{closureD}
 For any index
set $P = \{ p_1 < \dots < p_m \}\subset [1,2n+2]$, the Schubert variety $X_P$
is equal to the set of all $\Sigma\in Y_P$ such that for all
$c$ with $[c,n+1]\subset [P]$, we have $\dim(\Sigma\cap F_{c-1}) >
\#P\cap[1,c-1]$ or
  \begin{equation}\label{pareq}
    \dim((\Sigma+F_{c-1})\cap F_{n+1}) \ \equiv \ c-1 +\# P\cap [c,n+1] \ \
    \text{\rm(mod 2)}\,.
  \end{equation}
If $\type(P)=0$, then $X_P = Y_P$, while if $\type(P)>0$, then $X_P
\cup X_{\iota(P)} = Y_P$.  For any two index sets $P$ and $Q$, we have
$X_Q \subset X_P$ if and only if $Q\preceq P$.
\end{prop}
\begin{proof}
  Let $Z_P$ be the subset of $Y_P$ indicated in the proposition.
  Since $Z_P$ is $B$-stable and contains $X^\circ_P$, it suffices to
  show that $Z_P$ is closed, that $X^\circ_Q \subset Z_P$ implies $Q
  \preceq P$, and that $Q \preceq P$ implies $X_Q \subset X_P$.  The
  last of these assertions reduces to showing that, for $Q \prec P$,
  there exists an index set $P'$ such that $Q \preceq P' < P$ and
  $X_{P'} \subset X_P$.  Choose $j$ minimal such that $q_j < p_j$.
  Then $[q_j,p_j-1] \cap P = \emptyset$.
  
  Assume first that $[q_j,p_j-1] \not\subset [P]$, and set $x =
  \min([q_j,p_j-1] \ssm [P])$ and $P' = \{p_1,\dots,p_{j-1},x,
  p_{j+1},\dots,p_m\}$.  Then $Q \leq P' < P$, and the morphism $\bP^1
  \to \OG$ defined by (\ref{st1}) shows that $X^\circ_{P'} \subset
  X_P$.  If $Q \not\preceq P'$, then $(Q,P')$ is critical and
  $\type(Q) \neq \type(P')$.  This implies that $q_j \leq n$.  Let $c$
  be a critical index for $(Q,P')$.  Since $p_j \notin [P']$ we must
  have $p_j < c$ or $p_j \geq n+3$.  If $p_j < c$, then $c$ is also a
  critical index for $(Q,P)$ and $\type(P') = \type(P)$, which
  contradicts $Q \preceq P$.  On the other hand, if $p_j \geq n+3$,
  then $x=n+1 \in [P']$, $n \notin [P']$, and $P'\cap[1,n] \subsetneq
  Q\cap[1,n]$, which contradicts that $(Q,P')$ is critical.

  Otherwise we have $[q_j,p_j-1] \subset [P] \ssm P$.  It follows that
  $q_j \leq n$ and $p_j \leq n+2$.  Set $P' = (P \ssm \{p_j,N+1-q_j\})
  \cup \{q_j,N+1-p_j\}$.  Then $P' < P$, and the morphism defined by
  (\ref{st2}) shows that $X^\circ_{P'} \subset X_P$.  To see that $Q
  \preceq P'$, we first assume that $p_j \leq n+1$.  Then the argument
  from the odd orthogonal case shows that $Q \leq P'$, so if $Q
  \not\preceq P'$, then $(Q,P')$ is critical and $\type(Q) \neq
  \type(P')$.  Since $Q \preceq P$ and $\type(P') = \type(P)$, we
  deduce that $(Q,P)$ is not critical.  Let $c$ be a critical index
  for $(Q,P')$.  Since $c$ is not a critical index for $(Q,P)$ we must
  have $c \leq p_j$.  But then $[q_j,n+1] \subset [P]$, and since $Q
  \leq P$ we also have $[q_j,n+1] \subset [Q]$, so $q_j$ is a critical
  index for $(Q,P)$, a contradiction.  Finally, if $p_j=n+2$, then the
  above argument applied to $\iota(Q)$ and $\iota(P)$ shows that
  $\iota(Q) \preceq \iota(P')$, and therefore that $Q \preceq P'$.

  We next show that $Z_P$ is closed.  For each integer $c$ with
  $[c,n+1] \subset [P]$, we set $j_c = \#P\cap[1,c-1]$ and $U_c = \{
  \Sigma \in Y_P \mid \dim(\Sigma\cap F_{c-1}) = j_c\}$, and we let
  $Z_c \subset U_c$ be the subset of points satisfying (\ref{pareq}).
  Since all points $\Sigma \in Y_P$ satisfy $\dim(\Sigma\cap F_{c-1})
  \geq j_c$, it follows that $U_c$ is a Zariski open subset of $Y_P$.
  We claim that $Z_c$ is closed in $U_c$.  Let $\Sigma \in U_c$ be any
  point.  The condition $[c,n+1] \subset [P]$ implies that
  $\#P\cap[1,N+1-c] = j_c+n+2-c$, so $\dim(\Sigma\cap F^\perp_{c-1})
  \geq j_c+n+2-c$.  It follows that $\dim((\Sigma \cap F_{c-1}^\perp)
  + F_{c-1}) \geq n+1$.  But $(\Sigma \cap F_{c-1}^\perp)+F_{c-1} =
  (\Sigma + F_{c-1})\cap F_{c-1}^\perp$ is an isotropic subspace of
  $V$, so its dimension is exactly $n+1$.  Let $\OG_c =
  \OG_c(n-c+2,F_{c-1}^\perp/F_{c-1})$ denote the space of maximal
  isotropic subspaces of $F_{c-1}^\perp/F_{c-1}$, which has two
  connected components.  The claim now follows because $Z_c$ is the
  inverse image of one of these components under the morphism $U_c \to
  \OG_c$ defined by $\Sigma \mapsto ((\Sigma+F_{c-1})\cap
  F_{c-1}^\perp)/F_{c-1}$.  We conclude that $Z_P$ is closed using the
  identity
  \[ 
  Z_P = \bigcap_{[c,n+1] \subset [P]} ((Y_P \ssm U_c) \cup Z_c) \,.
  \]

  Finally, assume that $X^\circ_Q \subset Z_P$, and let $\Sigma \in
  X^\circ_Q$ be any point.  Since $X^\circ_Q \subset Y_P$, we have $Q
  \leq \ov{P}$.  It follows that $Q \leq P$, as otherwise $p_j=n+1$
  and $q_j=n+2$ for some $j$, in which case $\Sigma \in U_{n+1} \ssm
  Z_{n+1}$.  Moreover, if $c$ is any critical index for $(Q,P)$, then
  $\Sigma \in U_c$, and (\ref{pareq}) implies that
  $\type(Q)=\type(P)$.  It follows that $Q \preceq P$, which completes
  the proof.
\end{proof}

\begin{example}
  For the special Schubert varieties, indexed by typed $k$-strict
  partitions with a single non-zero part, Proposition~\ref{closureD}
  gives
  \[X_r = Y_r = \{ \Sigma \in \OG \mid \Sigma \cap F_{\epsilon(r)} \neq 0\},\]
  for $r \ne k$, where $\epsilon(r) = n+k+2-r$ if $r<k$, and
  $\epsilon(r) = n+k+1-r$ if $r>k$, while $X_k \cup X'_k = Y_k = \{
  \Sigma \in \OG \mid \Sigma \cap F_{n+2} \neq 0 \}$.  Let $\wt
  F_{n+1} \subset V$ be the unique maximal isotropic subspace such
  that $F_n \subset \wt F_{n+1} \neq F_{n+1}$.
  Proposition~\ref{closureD} gives
  \[
  X_k= \{\Sigma\in \OG\ |\ \Sigma\cap F_{n+1}\neq 0\} \ \text{ and } \
  X'_k= \{\Sigma\in \OG\ |\ \Sigma\cap \wt{F}_{n+1}\neq 0\},
  \]
  if $n$ is even, while the roles of $F_{n+1}$ and $\wt{F}_{n+1}$ are
  exchanged if $n$ is odd. This agrees with \cite[\S 3.2]{BKT1} and
  the assertions made in the introduction.
\end{example}


\begin{thebibliography}{BKT2}




\bibitem[BH]{BH} S. Billey and M. Haiman :
{\em Schubert polynomials for the classical groups},
J. Amer. Math. Soc. {\bf 8} (1995), 443--482.



\bibitem[BKT1]{BKT1} A. S. Buch, A. Kresch, and H. Tamvakis :
{\em Quantum Pieri rules for isotropic Grassmannians},
Invent. Math. {\bf 178} (2009), 345--405.


\bibitem[BKT2]{BKT2} A. S. Buch, A. Kresch, and H. Tamvakis :
{\em A Giambelli formula for isotropic Grassmannians},
Preprint (2010), available at arXiv:0811.2781.


\bibitem[BKT3]{BKT3} A. S. Buch, A. Kresch, and H. Tamvakis :
{\em Quantum Giambelli formulas for isotropic Grassmannians},
Math. Ann., to appear.


\bibitem[FK]{FK} S. Fomin and A. N. Kirillov :
{\em Combinatorial $B_n$-analogs of Schubert polynomials},
Trans. Amer. Math. Soc. {\bf 348} (1996), 3591--3620.

\bibitem[Fu]{Fu} W. Fulton :
{\em Determinantal formulas for orthogonal and symplectic degeneracy
loci}, J. Differential Geom. {\bf 43} (1996), 276--290.


\bibitem[FP]{FP} W. Fulton and P. Pragacz :
{\em Schubert varieties and degeneracy loci}, 
Lecture Notes in Math. {\bf 1689}, Springer-Verlag, Berlin, 1998.


\bibitem[Kr]{Kr} W. Kra\'skiewicz :
{\em Reduced decompositions in hyperoctahedral groups},
C. R. Acad. Sci. Paris S\'er. I Math. {\bf 309} (1989), 903--907.


\bibitem[KT]{KT} A. Kresch and H. Tamvakis :
{\em Double Schubert polynomials and degeneracy loci for the
classical groups}, Ann. Inst. Fourier {\bf 52} (2002), 1681--1727.


\bibitem[La]{La} T. K. Lam : {\em B and D analogues of stable Schubert
polynomials and related insertion algorithms}, Ph.D.\ thesis, M.I.T.,
1994; available at http://hdl.handle.net/1721.1/36537.


\bibitem[LS]{LS} A. Lascoux and M.-P. Sch\"{u}tzenberger :
{\em Polyn\^{o}mes de Schubert}, C. R. Acad. Sci. Paris S\'er. I
Math. {\bf 294} (1982), 447--450.


\bibitem[M]{M} I. Macdonald :
{\em Symmetric Functions and Hall Polynomials}, Second edition,
Clarendon Press, Oxford, 1995.


\bibitem[T1]{T1} H. Tamvakis : {\em Quantum cohomology of isotropic
Grassmannians}, Geometric Methods in Algebra and Number Theory,
311--338, Progress in Math. {\bf 235}, Birkh\"auser, 2005.

\bibitem[T2]{T2} H. Tamvakis : 
{\em Schubert polynomials and Arakelov theory of orthogonal flag 
varieties}, Math. Z. {\bf 268} (2011), 355-370.

\bibitem[T3]{T3} H. Tamvakis : {\em A tableau formula for eta polynomials}, 
Preprint (2011), arXiv:1109.6702.


\bibitem[Y]{Y} A. Young : 
{\em On quantitative substitutional analysis VI}, Proc. Lond. 
Math. Soc. (2) {\bf 34} (1932), 196--230.


\end{thebibliography}
\end{document}